\title{Ford Circles and Spheres}
\author{Sam Northshield\thanks{SUNY-Plattsburgh}}

\documentclass{article}

\usepackage{amsmath}
\usepackage{amsthm}
\usepackage{amssymb}
\usepackage{url}
\usepackage{graphicx}

\numberwithin{equation}{section}
\newtheorem{thm}{Theorem}[section]
\newtheorem{theorem}[thm]{Theorem}

\newtheorem{lemma}[thm]{Lemma}

\newtheorem{proposition}[thm]{Proposition}

\newtheorem{corollary}[thm]{Corollary}
\newtheorem{problem}[thm]{Problem}

\begin{document}
\date{}
\maketitle
 \begin{abstract}
Ford circles are parameterized by the rational numbers but are also the result of an iterative geometric procedure.   We review this and introduce an apparently new parameterization by solutions of a certain quadratic Diophantine equation.   We then generalize to Eisenstein and Gaussian rationals where the resulting ``Ford spheres" are also the result of iterative geometric procedures and are also parameterized by solutions of certain quadratic Diophantine equations.  We generalize still further to imaginary quadratic fields. \end{abstract}

\section{Introduction}
The set of Ford circles form an arrangement of circles each above but tangent to the $x$-axis at a rational number, with disjoint interiors,  that is maximal in the sense that no additional such circles can be added.   

\begin{figure}[htbp] 
 \centering\includegraphics[ width=2 in]{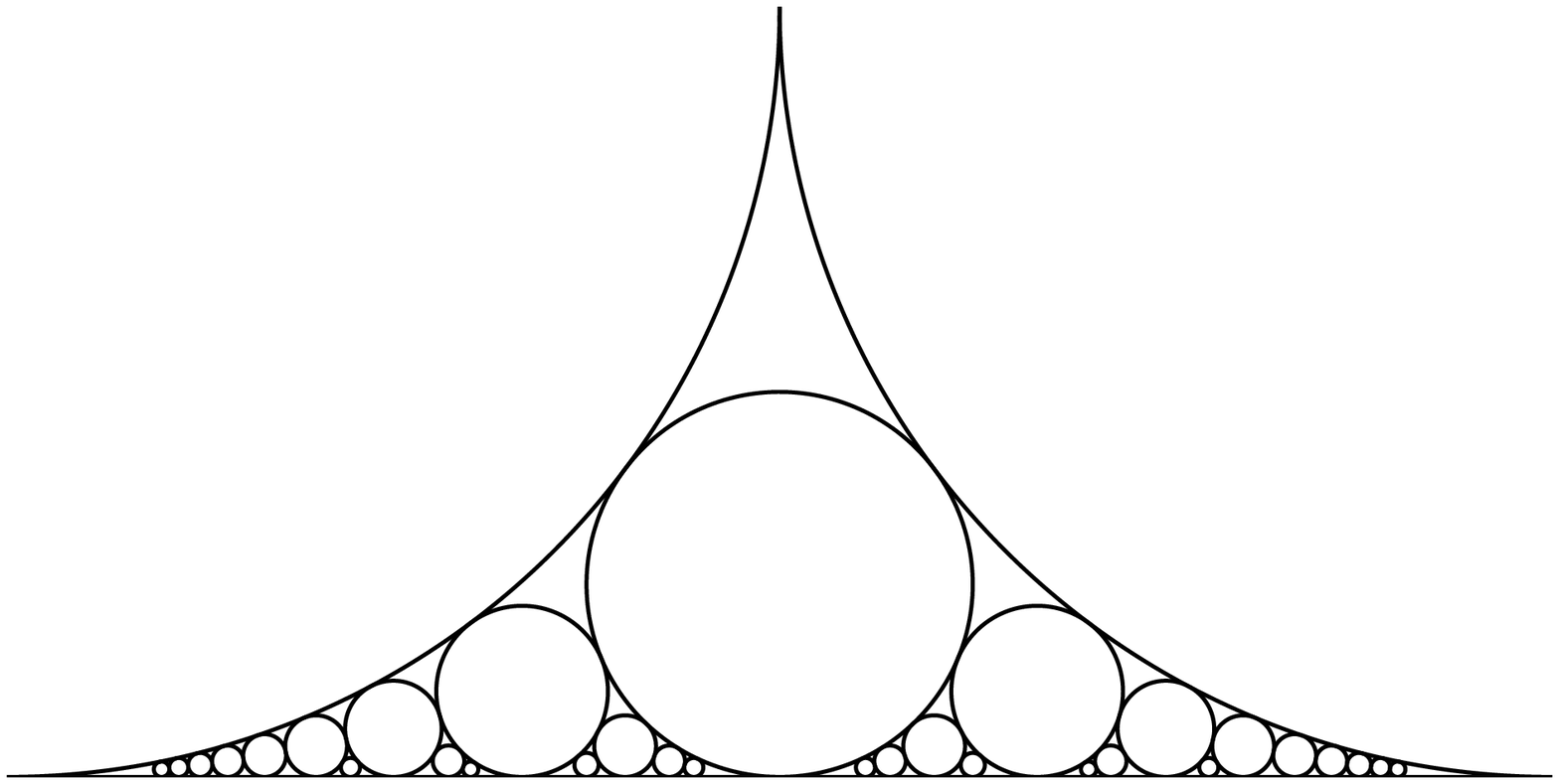} 

   \centering{
   \includegraphics[ width=.8 in]{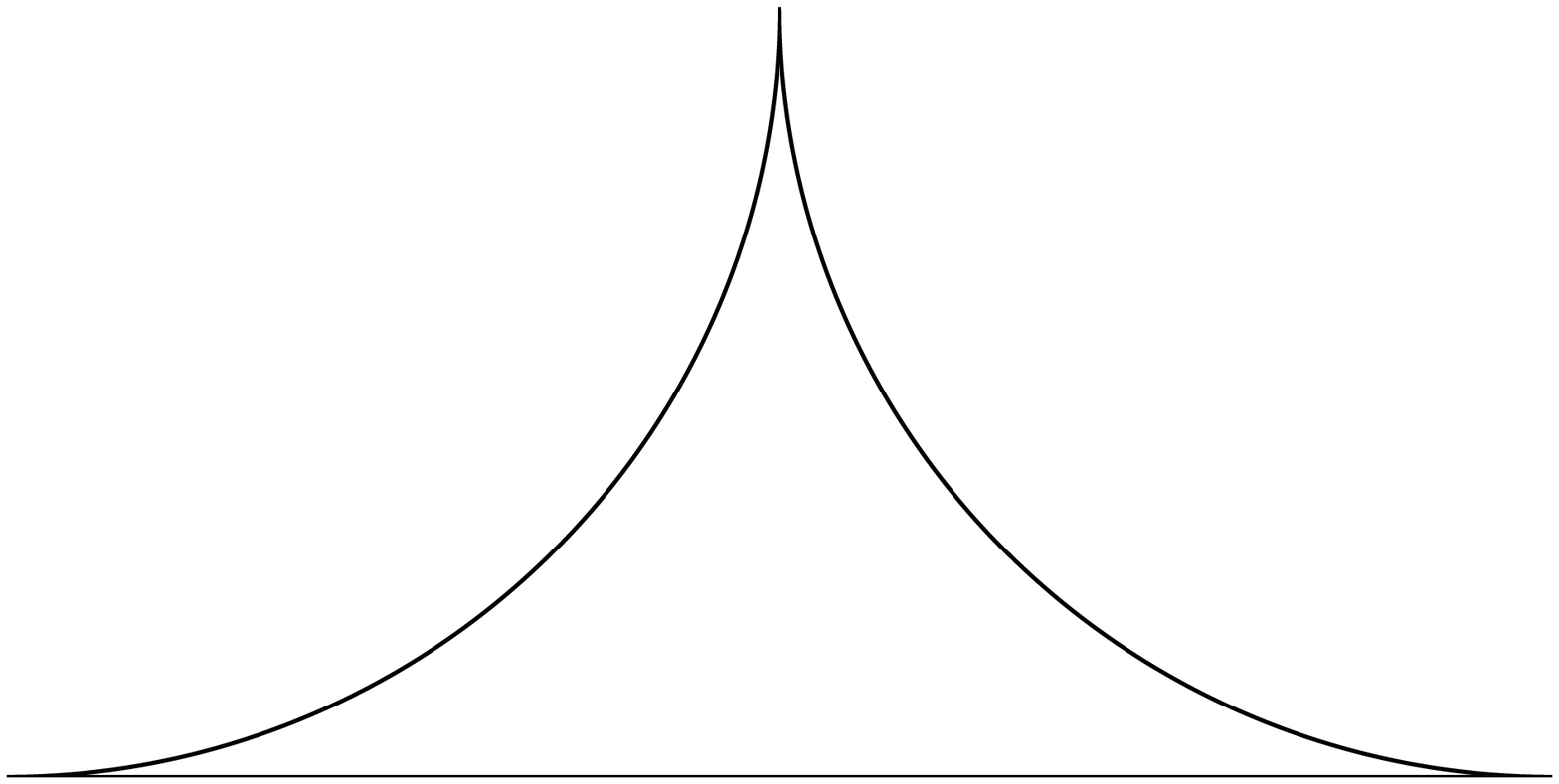} 
   \includegraphics[ width=.8 in]{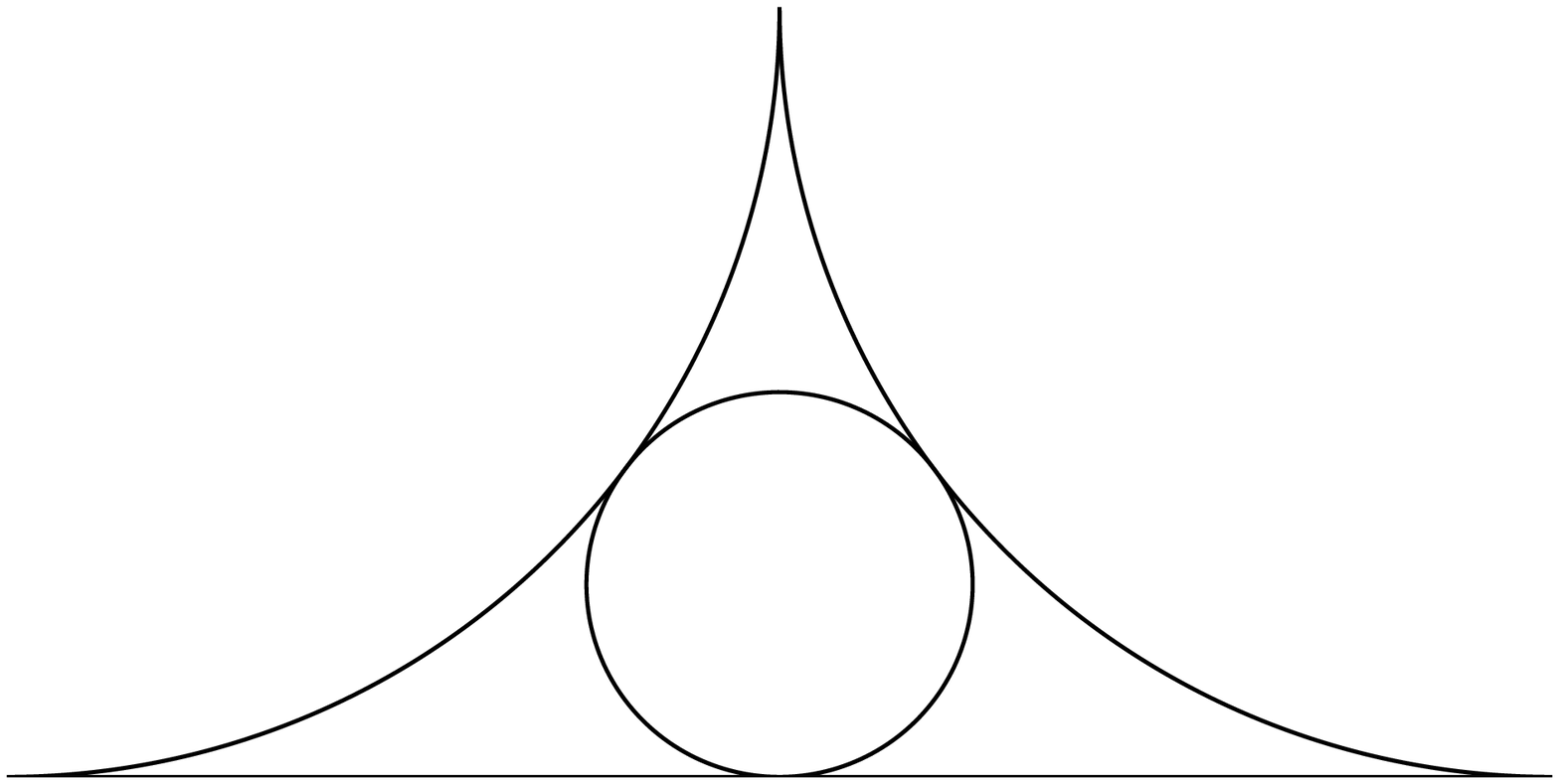} 
   \includegraphics[ width=.8 in]{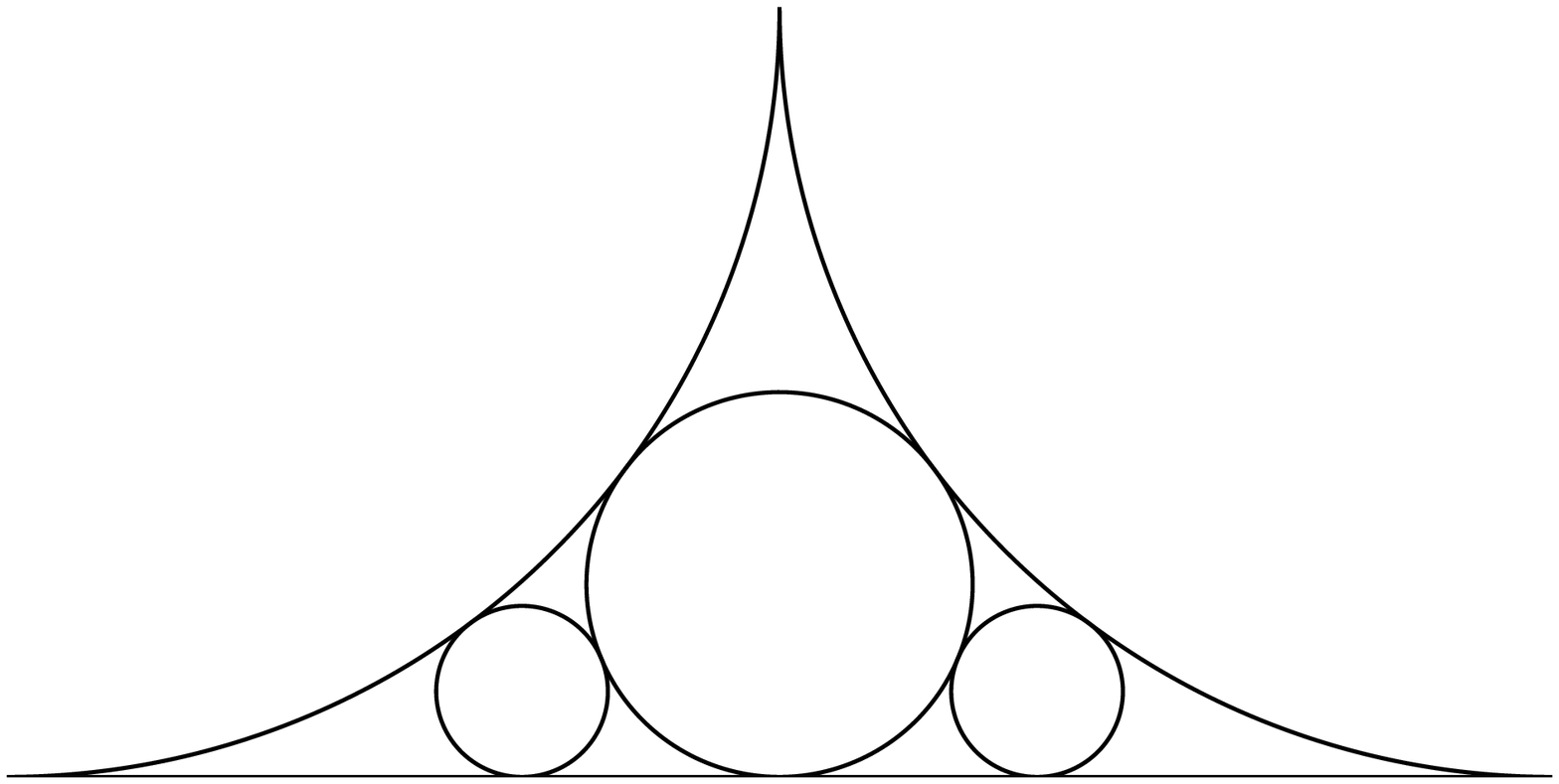}
    \includegraphics[ width=.8 in]{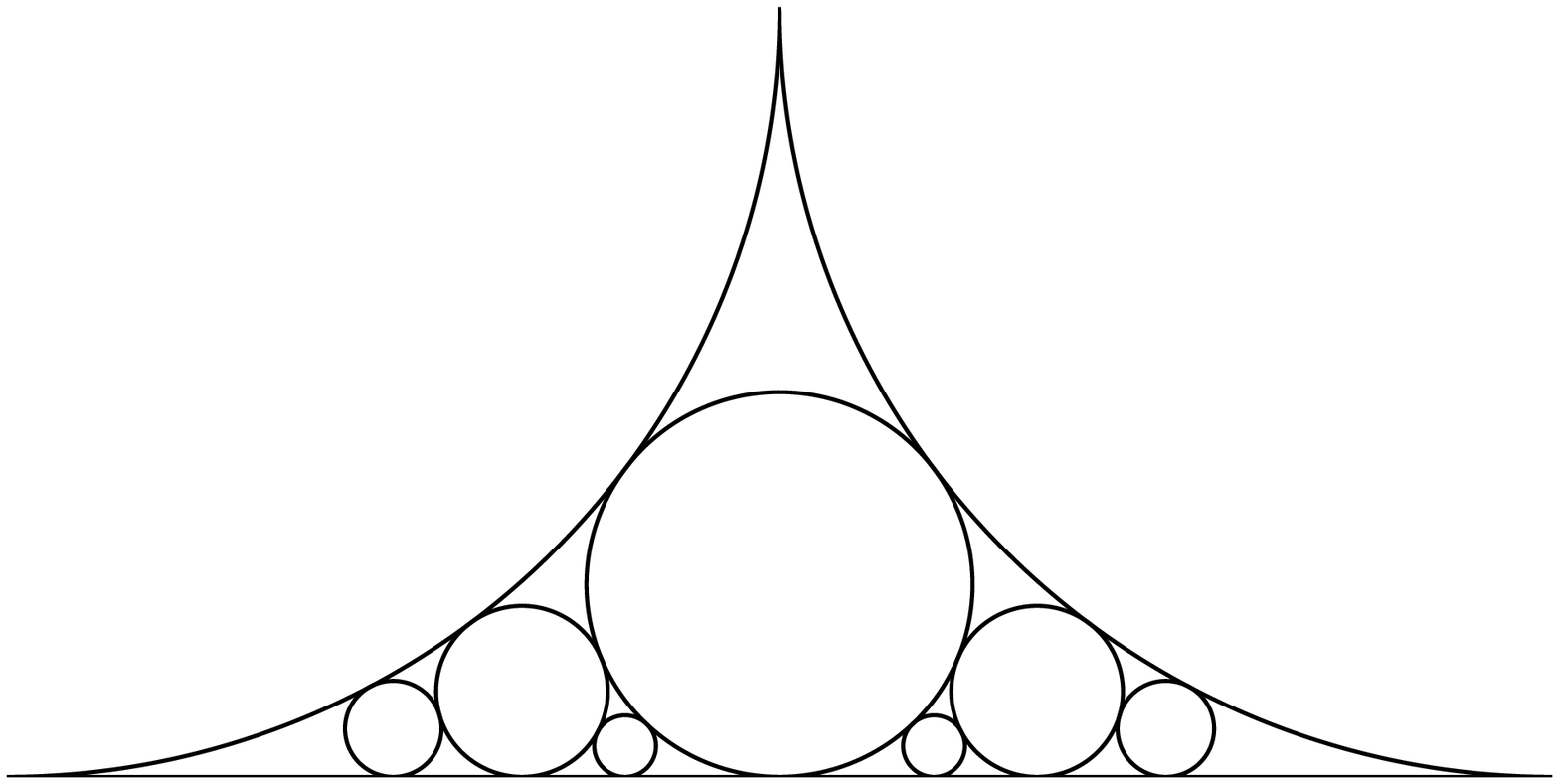} 
    }
       \caption{Ford circles and their geometric construction}
   \label{fig. 7}
\end{figure}

They provide a natural way of visualizing the Diophantine approximation of real numbers by rationals,  are used in the ``circle method" of Ramanujan and Hardy, and give a geometric way of looking at continued fractions.  They form part of an Apollonian circle packing,  a topic of intense current study, as well as provide some connection to some older open problems (e.g., see \cite{NorSt} for connection to the Riemann hypothesis,  \cite{GLMWYnt} for Hausdorff dimension).  Three dimensional analogues are then certainly of interest.  

In Section 2, we construct the set of Ford circles in three different ways.   First, to each rational number $a/b$ (in lowest terms), assign a circle above but tangent to the $x$-axis at $a/b$ with radius$1/2b^2$.   The set of circles thus parameterized by ${\Bbb Q}$ is the set of Ford circles which we denote by ${\cal P}$.   A geometric construction of this set starts with circles of radius 1/2 above but  tangent to the $x$-axis at each of the integers.  We then proceed inductively:  given any two circles tangent to each other, we add the unique circle between and tangent to those two (see Figure 1).  The maximal collection formed by this procedure is denoted ${\cal G}$.       Third, for every relatively prime integer solution $(a,b,c)$  of  $a^2+b^2+c^2=(a+b+c)^2$  we define a new a circle above and tangent to the $x$-axis at $b/(a+b)$ with radius $1/2(a+b)$.   In this last construction, the tangent point has ``projective barycentric coordinates" $(a,b)$ and so we call this last construction barycentric and denote the collection by  ${\cal B}$.  These three collections of circles are shown to be all the same.   

In Section 4, we define a family of Ford spheres and three parameterizations of it that are analogous to the constructions of Ford circles.   For $\omega$ the cube root of unity with positive imaginary part, we may parameterize a family of spheres by Eisenstein rationals ${\Bbb Q}(\omega)$, as we did for Ford circles:  given two relatively prime Eisenstein integers $\alpha,\beta$, we construct a sphere in ${\Bbb C}\times{\Bbb R^+}$  tangent to ${\Bbb C}$ at $\alpha/\beta$ with radius $1/2|\beta|^2$.  
A geometric construction of this family of spheres is then given.   Starting with spheres of radius 1/2 at every point in the triangular lattice $\{m+n\omega\:  m,n\in{\Bbb Z}\}$, iterate the process:  given any three mutually tangent spheres, add the 
two uniquely determined spheres each of which is tangent to the original three and to the complex plane.  
This process is ``tetrahedral" in the sense that for every three mutually tangent spheres (i.e., their contact graph is a triangle), there is a fourth so that the contact graph of all four is a tetrahedron.     Finally, the relatively prime integer solutions  of $(a+b+c+d)^2=a^2+b^2+c^2+d^2$ parameterize these spheres barycentrically. 

In Section 5, we define another family of Ford spheres and three parameterizations of it that are based on the Gaussian integers.   That is, we construct family of spheres as in Section 3 using relatively prime Gaussian integers instead of Eisenstein integers.   There is a corresponding geometric construction, this time  based on octahedra instead of tetrahedra.   Finally, the barycentric parameterization is in now terms of solutions of $(a+b+c+d)^2=2(a^2+b^2+c^2+d^2).$  

In Section 6, we consider a definition of Ford spheres parameterized by ${\Bbb Q}(\sigma)$ where ${\Bbb Z}[\sigma]={\cal O}_{{\Bbb Q}(\sqrt{-D})}$ for some other positive integers $D$ besides 1 and 3 (so that ${\Bbb Z}[\sigma]$ is a UFD).  In this situation, we develop a barycentric construction as well.   In the three cases where ${\Bbb Z}[\sigma]$ is a Euclidean domain that are not covered by Sections 3 and 4, we conjecture a geometric construction for the corresponding Ford spheres.   

In Section 7, we present a summary of the paper and some directions of further research.

The Ford spheres of associated with $\omega$ (Section 4) have been considered by Rieger \cite{R1, R2},   and those associated with $i$  (Section 5)  by Pickover \cite{Pick}.   There is doubtless some overlap with these papers as well as with the material covered in a series of papers by Asmus Schmidt \cite{Sch1,Sch2,Sch3, Sch4}.   These last four papers are concerned with Diophantine approximation and geometry relate to the cases where  $D=1,3,2,11$ respectively.   Sullivan \cite{Su} deals with quite general arrays of spheres ${\Bbb C}\times{\Bbb R^+}$  tangent to ${\Bbb C}$.   

Some of the novelties that appear in this paper are as follows.   The ``barycentric" parameterizations of Sections 2,4,5 and 6 in terms of solutions of Diophantine equations (e.g.,Equations 7, 8, 11) is new.   We show that for such solutions $(a,b,c,d)$ of (7),  $|a+b+c|=|\gamma|^2$ for some Eisenstein integer $\gamma$ [Cor. 4.6].   Similarly, for solutions $(a,b,c,d)$ of (8),  $|a+b|$ is a sum of two integer squares and $|a+b+c|$ is the sum of norm squares of two Eisenstein integers [Cor. 5.10].   We introduce a quadratic form on pairs of solutions of equation 7 which takes on the value 1 if and only if the corresponding spheres are tangent [Th. 3.4].   Also in Section 3, we are able to redefine the Poincar\'e extension of a M\"obius transformation in terms of its action on spheres [Prop. 3.2].   In Section 6, the Ford spheres parameterized by ${\Bbb Q}(\sigma)$  are introduced.   It is shown also these spheres can be parameterized barycentrically (in terms of solutions of equation (11)).  An intriguing connection is made between these solutions and a group related to ``secant addition" of \cite{NorSecant}.

\section{Ford circles}

We say a circle in the $x,y$-plane is \emph{normal} if it is above and tangent to the $x$-axis.  
For $t\in{\Bbb R}$ and $r>0$, let $C(t,r)$ be the circle with center $(t,r)$ and radius $r$.    Hence, $C(t,r)$ is normal.  We note that every normal circle can be uniquely represented as $C(t,r)$ for some $t,r$.    By the Pythagorean theorem, two circles $C(t,r)$ and $C(t',r')$ are tangent (we write $C(t,r)||C(t',r')$) if and only if 
$$(t-t')^2+(r-r')^2=(r+r')^2$$ or, equivalently,
$$(t-t')^2=4rr'.\eqno{(1)}$$

Given $a,b\in{\Bbb R}$ with $b>0$, we define $$C_{a,b}:=C\left(\frac ab, \frac 1{2b^2}\right).$$
Then every  circle above and tangent to the $x$-axis can be uniquely represented as $C_{a,b}$ for some real $a,b$ ($b>0$):
$$C(t,r)=C_{t/\sqrt{2r}, 1/\sqrt{2r}}.$$
   By (1), two such circles are tangent, i.e., 
$C_{a,b}||C_{c,d}$,  if and only if $|ad-bc|=1$.  

From this point on, we write $a\perp b$ for $a,b$ relatively prime.  We define the set of Ford circles:
$${\cal P}:=\{C_{a,b}:  a,b\in{\Bbb Z}, a\perp b\}.$$

We define a set of circles to be \emph{normal} if each circle is normal and no two circles have intersecting interiors (so, of course, ${\cal P}$ is normal).   We order the set of normal sets of circles:  
$${\cal A}<{\cal B} \Longleftrightarrow \bigcup_{A\in {\cal A}}A\subset\bigcup_{B\in {\cal B}}B.$$

\begin{proposition}   ${\cal P}$ is maximal with respect to the order $<$.  \end{proposition}
\begin{proof}  We first show that no circle in ${\cal P}$ can be enlarged in the sense of having its radius increased but its tangent point to the $x$-axis the same. 
Note that any two Ford circles have disjoint interiors (since $|ad-bc|\ge 1$ for two Ford circles $C_{a,b}$ and $C_{c,d}$).  Given a Ford circle $C_{a,b}$,  $a\perp b$ and so there exist $c,d$ ($d>0$) such that $|ad-bc|=1$ and thus there exists another Ford circle, $C_{c,d}$, tangent to $C_{a,b}$.   This implies that no circle in ${\cal P}$ can be enlarged and still maintain normality (i.e., its interior is still disjoint from the interiors of all the other circles in ${\cal P}$).

Next, we show that no normal circle  can be added.    Suppose there exist $x,r$ so that  for all Ford circles $C_{a,b}$, $C(x,r)^{\circ}\cap C_{a,b}= \emptyset$.  Obviously $x$ must be irrational and thus $nx \mod 1$ is dense in $[0,1]$.   Consequently, the set $\{bx-a: a,b\in{\Bbb Z}, b>0\}$ is dense in ${\Bbb R}$ and so, for any fixed $r$, there exist $a,b$ such that 
$$\left|x-\frac ab\right|<\frac{\sqrt{2r}}b.$$
By equation (1), this implies $C(x,r)^{\circ}\cap C_{a,b}^{\circ}\neq\emptyset$, a contradiction.  \end{proof}

We shall use the following geometric fact.  

\begin{proposition} Given two tangent normal circles (tangent to the $x$-axis at, say, $x$ and $y$), there is a unique third normal circle tangent to both and to the $x$-axis at a point between $x$ and $y$. \end{proposition}

\begin{proof}  Suppose $C(x,r)||C(y,s)$ where $x<y$.   For $t\in(x,y)$, let $R_1(t):=(t-x)^2/4r$ and $R_2(t)=(t-y)^2/4s$.   Then $C(t,R_1(t))||C(x,r)$ and $C(t,R_2(t))||C(y,s)$ (and each is unique).   As $t$ increases,  $R_1(t)$ increases from 0 and $R_2(t)$ decreases to 0 and so there exists a unique $z\in(x,y)$ such that $R_1(z)=R_2(z)$.   Hence $C(z,R_1(z))$ is tangent to both $C(x,r)$ and $C(y,s)$.  \end{proof}

We call this new circle the \textit{child} of the other two circles.  It is easy to see that if $C_{a,b}||C_{c,d}$ then  $C_{a+c,b+d}$  and $C_{a-c,b-d}$ are tangent to both $C_{a,b}$ and $C_{c,d}$.
We say that $C_{a,b}$ and $C_{c,d}$ are \emph{parents}  of $C_{a+c,b+d}$ and $C_{a-c,b-d}$

It is possible to find the parents of a given Ford circle. For an ordered pair $(a,b)$ of positive integers , consider the \emph{slow Euclidean algorithm}:
$$[a,b]\longmapsto \begin{cases} [a,b-a] &\text{if $a<b$,}\\
[a-b,b] &\text{if $a>b$,}\\ \text{stop}&\text{if
$a=b$}.\end{cases}$$ 
This algorithm must terminate (since the sum of the two entries is positive and strictly decreasing) and, since the greatest common divisor is preserved at each step, this algorithm terminates with 
$[\gcd(a,b),\gcd(a,b)]$.  For example:
$$[14,5]{\buildrel L\over\longmapsto}[9,5]{\buildrel L\over\longmapsto}[4,5]{\buildrel R\over\longmapsto}[4,1]
{\buildrel L\over\longmapsto}[3,1]{\buildrel L\over\longmapsto}[2,1]{\buildrel L\over\longmapsto}[1,1];\eqno{(2)}$$
here we labeled each arrow according to which of the two entries is changed.   

This leads to a definition of maps $L(a,b):=(a-b,b)$ and $R(a,b):=(a,b-a)$ so that, for example, $L\circ L\circ L\circ R \circ L \circ L(14,5)=(1,1)$.   The maps $L$ and $R$ are invertible and thus $L^{-1}\circ L^{-1}\circ R^{-1}\circ L^{-1}\circ L^{-1}\circ L^{-1}(1,1)=(14,5)$.  

In general, every relatively prime pair of positive integers $(a,b)$ gives rise to a word $w_1w_2...w_n$ in $\{L,R\}^*$ so that
$w_1^{-1}\circ w_2^{-1}\circ ... \circ w_n^{-1}(1,1)=(a,b).$
It follows easily by induction that if
$(x,y):=w_1^{-1}\circ w_2^{-1}\circ ... \circ w_n^{-1}(0,1)$ and $(u,v):=w_1^{-1}\circ w_2^{-1}\circ ... \circ w_n^{-1}(1,0)$, then 
$C_{x,y}$ and $C_{u,v}$ are the parents of $C_{a,b}$.    We call this the \emph{parent algorithm}.

We now give a geometric construction of a normal set of circles that turns out to coincide with ${\cal P}$.
If ${\cal A}$ and ${\cal B}$ are normal sets of circles, we write  ${\cal A}{\lessdot} {\cal B}$ if ${\cal B}={\cal A}\cup\{C\}$ where $C$ is the child of two circles in ${\cal A}$ (as constructed in Proposition 2.2).  
Starting with the normal set of circles ${\cal G}_0:=\{C_{n,1}:  n\in{\Bbb Z}\}$, let ${\cal G}$ be the union of all normal sets of circles ${\cal G}'$ that are maximal elements of a finite chain ${\cal G}_0\lessdot\dots\lessdot{\cal G}'$.   

\begin{lemma}  ${\cal G}={\cal P}$.  \end{lemma}

\begin{proof}  Note that $C_{-a,b}$ has parents $C_{-x,y}$ and $C_{-u,v}$ if $C_{a,b}$ has parents $C_{x,y}$ and $C_{u,v}$.
Hence the existence of 
the  parent algorithm shows that  any Ford circle $C_{a,b}$ is either in ${\cal G}_0$ (if $b=1$) or $C_{a,b} $ has two ``parents" (if $b\neq1$)  Hence, by an induction argument, every $C_{a,b}$ is in ${\cal G}$ and thus, by the maximality of ${\cal P}$, the theorem follows.    \end{proof}

%Recall that given $n+1$ points $P_0,...,P_n$ in general position in ${\Bbb R}^n$ (i.e., for which the vectors $\{P_j-P_0:  j\ge 1\}$ are linearly independent), every point $P$ in ${\Bbb R}^n$ can be represented uniquely as a linear combination of $c_0P_0+...+c_nP_n$ where $c_0+...+c_n=1$.   The coefficients $c_j$ are called the \emph{barycentric coordinates} of $P$.   

 We introduce a new way to parameterize normal circles.   Given $a,b\in{\Bbb R}$ (with $a+b>0$), let 
$$\langle a,b\rangle:=C\left(\frac{b}{a+b}, \frac1{2(a+b)}\right).$$
It is easy to verify that every normal circle is represented in the form $\langle a,b\rangle$:
$$C(x,r)=\left\langle\frac{1-x}{2r},\frac x{2r}\right\rangle$$
   and therefore every normal circle is represented in this new way.   We refer to such representations as ``barycentric" and will explain this terminology in Section 3. 
   
   The set of Ford circles can be represented barycentrically in terms of solutions of a certain Diophantine equation.
Let
$${\cal B}:=\{\langle s,t \rangle:  s,t,u\in {\Bbb Z},  \gcd(s,t,u)=1, s+t>0, (s+t+u)^2=s^2+t^2+u^2\}.$$

\begin{theorem}  ${\cal B}={\cal P}={\cal G}$. \end{theorem}
\begin{proof}  Given a Ford circle $C_{a,b}$, let $s:=b^2-ab$, $t:=ab$, and $u:=a^2-ab$.  
Note that $\langle s,t\rangle=C_{a,b}$ and it is easy to verify that $\langle s,t\rangle\in{\cal B}$.
Hence ${\cal P}<{\cal B}$ and, by the maximality of ${\cal P}$,  and Lemma 2.3, the theorem follows.  \end{proof}

A surprising result follows (see \cite{MN}).  

\begin{corollary}  If $(a,b,c)$ is a relatively prime integer solution of $a^2+b^2+c^2=(a+b+c)^2$  then $|a+b|$ is a perfect square.  \end{corollary}

\begin{proof}  If $(a,b,c)$ satisfies the hypothesis, then $\langle a,b\rangle\in{\cal B}$.  By Theorem 2.4,  $\langle a,b\rangle=C_{m,n}$ for some $m,n$.   The radii of these two circles are equal and therefore $|a+b|=n^2$.   \end{proof}

The results of this Section, with different proofs, have appeared in a paper \cite{MN} by the author and one of his students.

\section{Spheres, in general}

We identify ${\Bbb R}^3$ with ${\Bbb C}\times {\Bbb R}$.   For $z\in {\Bbb C}$ and $r>0$, let
$S(z,r)$ be the sphere with center $(z,r)$ and radius $r$.   This sphere can be visualized as the sphere above and tangent to the complex plane at $z$ with radius $r$.    We say that a sphere is \emph{normal} if it is of the form $S(z,r)$ for some $z\in {\Bbb C}$ and $r>0$.
As in equation (1), it is easy to verify that $S(z,r)$ and $S(w,s)$ are tangent (we write $S(z,r)||S(w,s)$) if and only if
$$|z-w|^2=4rs.\eqno{(3)}$$

Given three points in ${\Bbb C}$, it turns out that there is a unique set of three mutually tangent normal spheres tangent to ${\Bbb C}$ at those points.

\begin{proposition} $S(P_1,r_1)$,   $S(P_2,r_2)$, and $S(P_3,r_3)$ are mutually tangent if and only if 
$$r_1=\dfrac {|P_1-P_2|\cdot |P_1-P_3| }{2|P_2-P_3|}, r_2=\dfrac {|P_1-P_2|\cdot |P_2-P_3| }{2|P_1-P_3|}, r_3=\dfrac {|P_1-P_3|\cdot |P_2-P_3| }{2|P_1-P_2|}.$$  Consequently, given three points $P_1,P_2,P_3$, there exist unique real numbers $r_1,r_2,r_3>0$ such that $S(P_1,r_1)$,   $S(P_2,r_2)$, and $S(P_3,r_3)$ are mutually tangent.
 \end{proposition}

\begin{proof}   Suppose the spheres $S(P_1,r_1)$,   $S(P_2,r_2)$, and $S(P_3,r_3)$ are mutually tangent.   By (3), $|P_1-P_2|^2=4r_1r_2$, $|P_1-P_3|^2=4r_1r_3$, and $|P_2-P_3|^2=4r_2r_3$.  Hence
$|P_1-P_2|^2/|P_1-P_3|^2=r_2/r_3$ and so 
$$\dfrac{|P_1-P_2|^2}{|P_1-P_3|^2}\cdot |P_2-P_3|^2=\dfrac{r_2}{r_3}\cdot 4r_2r_3=4r_2^2.$$
It follows that $r_2=\dfrac {|P_1-P_2|\cdot |P_2-P_3| }{2|P_1-P_3|}$.  Similarly, $r_1=\dfrac {|P_1-P_2|\cdot |P_1-P_3| }{2|P_2-P_3|}$ and $ r_3=\dfrac {|P_1-P_3|\cdot |P_2-P_3| }{2|P_1-P_2|}$.  

Conversely, suppose $r_1=\dfrac {|P_1-P_2|\cdot |P_1-P_3| }{2|P_2-P_3|}$, $r_2=\dfrac {|P_1-P_2|\cdot |P_2-P_3| }{2|P_1-P_3|}$, and $r_3=\dfrac {|P_1-P_2|\cdot |P_2-P_3| }{2|P_1-P_3|}$.    Then 
$4r_1r_2=|P_1-P_2|^2$ and so, by (3),   $S(P_1,r_1)||S(P_2,r_2)$.   Similarly,  $S(P_1,r_1)||S(P_3,r_3)$ and  $S(P_3,r_3)||S(P_2,r_2)$.   \end{proof}

It is well known that M\"obius transformations take circles to circles (straight lines are considered circles too since, on the Riemann sphere, they are circles through $\infty$;  see \cite{Ne}).   We shall use the notation
$$\begin{pmatrix}
a&b\\
c&d\\
\end{pmatrix}(z):=\dfrac{az+b}{cz+d}.$$

Every M\"obius transformation ${\Bbb C}\rightarrow {\Bbb C}$ extends to a unique sphere-preserving continuous map ${\Bbb C}\times [0,\infty)\rightarrow {\Bbb C}\times [0,\infty)$ known as its Poincar\'e extension.   Its formula is available in \cite{Beardon}, for example.    We shall use the fact that such an extension exists (even without needing its explicit formula) to express it in terms of normal spheres. Given a M\"obius transformation $m(z)$,  extend $m$ to normal spheres by  defining 
$$\hat m: S(z,r)\longmapsto S(m(z), |m'(z)|r). \eqno{(4)}$$
Since 
$$|m(z)-m(w)|^2=|m'(z)||m'(w)||z-w|^2,$$
it follows that
if  $S(z,r)||S(w,s)$ then $\hat m(S(z,r))||\hat m(S(w,s))$.     That is, $\hat m$ preserves tangency of spheres.  

The Poincar\'e extension $\overline m$ of $m$, since it takes normal spheres to normal spheres, also acts as a function on the set of normal spheres:  for some $r'$ depending on $z$ and $r$,
$$\overline m(S(z,r))=S(m(z), r').$$   
Given $(w,t)\in  {\Bbb C}\times {\Bbb R}^+$,  $(w,t)$ is a point of interSection of two normal spheres,  say $S(z_1,r_1)$ and $S(z_2, r_2)$.   Let $S(z_3,r_3)$ be an arbitrarily chosen normal sphere tangent to both  $S(z_1,r_1)$ and $S(z_2, r_2)$.    Since $\{\overline m(S(z_i,r_i)): i=1,2,3\}$ and $\{\hat m(S(z_i,r_i)): i=1,2,3\}$ have the same three points of tangency to ${\Bbb C}$,  Proposition 3.1 implies $\overline m(S(z_i,r_i))=\hat m(S(z_i,r_i))$ for $i=1,2,3$ and therefore
$$ \hat m(S(z_1,r_1))\cap  \hat m(S(z_2,r_2))=\{\overline m(w,t)\}.$$
That is, the Poincar\'e extension $\overline m$ of $m$ is actually defined by (4).  

We now describe two ways to parameterize normal spheres.  First, for $\alpha,\beta\in{\Bbb C}$, let
$$S_{\alpha,\beta}:=S\left(\frac {\alpha}{\beta}, \frac 1{2|\beta|^2}\right).$$
Clearly, every normal sphere can be uniquely represented thus:
$$S(z,r)=S_{z/\sqrt{2r},1/\sqrt{2r}}.$$
We interpret $S_{1,0}$ to be the plane parallel, and one unit above, ${\Bbb C}$.
The Poincar\'e extension is easily seen to have a simple formulation in terms of these spheres:

\begin{proposition}  If $\Delta:=\sqrt{|ad-bc|}$ then 
$$\begin{pmatrix}
a&b\\
c&d\\
\end{pmatrix}(S_{\alpha,\beta}):=S_{(a\alpha+b\beta)/\Delta, (c\alpha+d\beta)/\Delta}.$$\end{proposition}

\begin{proof}  Since $$\begin{pmatrix}
a&b\\
c&d\\
\end{pmatrix}'(z)=\dfrac{ad-bc}{(cz+d)^2},$$
$$\begin{aligned} S_{(a\alpha+b\beta)/\Delta, (c\alpha+d\beta)/\Delta}&=S\left(\dfrac{a\alpha+b\beta}{c\alpha+d\beta},\dfrac{|ad-bc|}{2|c\alpha+d\beta|^2}\right)\\
&=S\left(\begin{pmatrix}
a&b\\
c&d\\
\end{pmatrix}\left(\dfrac{\alpha}{\beta}\right),\left|m'\left(\dfrac{\alpha}{\beta}\right)\right|\dfrac 1{2|\beta|^2}\right)\\
&=\begin{pmatrix}
a&b\\
c&d\\
\end{pmatrix}\left(S\left(\dfrac{\alpha}{\beta},\dfrac {1}{2|\beta|^2}\right)\right)=\begin{pmatrix}
a&b\\
c&d\\
\end{pmatrix}(S_{\alpha,\beta}).\end{aligned}$$\end{proof}

A fact that will be used in the next Section is the following.

\begin{proposition}  Given any three mutually tangent normal spheres, there are exactly two normal spheres so that  each one,  with the original three,  form four mutually tangent spheres (with the understanding that planes parallel to the complex plane are spheres tangent at $\infty$).   \end{proposition}

\begin{proof}  Recall $\omega:=(-1+i\sqrt 3)/2$ is a cube root of unity and that $0, 1, 1+\omega$ form the vertices of an equilateral triangle.   The spheres $S_{0,1}, S_{1,1}, S_{1+\omega, 1}$ form three mutually tangent spheres each of radius 1/2.   Clearly the plane $z=1$ parallel to the complex plane (also denoted $S_{1,0}$) is tangent to all the original three.   The locus of tangent points of spheres that are tangent to two spheres chosen from $S_{0,1}, S_{1,1}, S_{1+\omega, 1}$ form a straight line and so there is a unique sphere tangent to ${\Bbb C}$ at the interSection of three straight lines that is tangent to $S_{0,1}, S_{1,1}, S_{1+\omega, 1}$.   Therefore the proposition holds for the three spheres.

Now, suppose we have three mutually tangent spheres tangent at, say, $z_1, z_2, z_3$ respectively.   There is a M\"obius transformation taking these three points to $0,1,1+\omega$ respectively and so, by
Proposition 3.1, the Poincar\'e extension takes the spheres to $S_{0,1}, S_{1,1}, S_{1+\omega, 1}$.    Any sphere tangent to all three also get mapped to either $S_{1,0}$ or $S_{1,1-\omega}$ and therefore the inverse map applied to $S_{1,0}$ and $S_{1,1-\omega}$ gives all of the spheres tangent to those at $z_1, z_2, z_3$ . 
\end{proof}

Recall the concept of ``barycentric coordinates".  
Given three non-colinear points in the plane (we henceforth define $P_0=0, P_1=1, P_2=1+\omega$ so that the three points form an equilateral triangle of side length 1), it is possible to express every point in the plane uniquely as a real convex combination of the three:
 $z=aP_0+bP_1+cP_2,  a+b+c=1$.   In this case, we say that $z$ has barycentric coordinates $(a,b,c)$.   
 Removing the condition $a+b+c=1$, we say that $z$ has ``projective barycentric coordinates" $(a,b,c)$ if 
 $z=(aP_0+bP_1+cP_2)/(a+b+c)$.     It is then possible to describe every normal sphere uniquely in terms of terms of three real numbers $a,b,c$ where $a+b+c>0$:
  $$\langle a,b,c\rangle:= S\left(\frac{aP_1+bP_2+cP_3}{a+b+c}, \frac1{2(a+b+c)}\right).$$
 
 We refer this parameterization as ``barycentric" and note that every normal sphere has a barycentric representation:  for real $x,y,r$ with $r>0$, 
 $$S(x+iy,r)=\left\langle\frac{1-x-y/\sqrt 3}{2r}, \frac{x-y/\sqrt 3}{2r}, \frac{2y/\sqrt 3}{2r}\right\rangle.$$

The barycentric parameterization allows a test for tangency.   Consider
the bilinear form 
$$Q({\bf u},{\bf v}):=({\bf u}\cdot {\bf 1})({\bf v}\cdot {\bf 1})-{\bf u}\cdot{\bf v}$$
where ${\bf u},{\bf v}\in {\Bbb R}^4$ and ${\bf 1}:=(1,1,1,1)$.  Let 
$${\cal U}: =\{(a,b,c,d)\in{\Bbb R}^4:  (a+b+c+d)^2=a^2+b^2+c^2+d^2, a+b+c>0\},$$
and, for ${\bf u}:=(a,b,c,d)$, let $\langle {\bf u}\rangle$ be the sphere $\langle a,b,c\rangle$.

%be the defining set for ${\cal B}$.  Then ${\cal U}$ is part of the zero set of the bilinear form 
%$$Q({\bf u},{\bf v}):=({\bf u}\cdot {\bf 1})({\bf u}\cdot {\bf 1})-{\bf u}\cdot{\bf v}$$
%where ${\bf 1}:=(1,1,1,1)$.  

\begin{theorem} For ${\bf u,v}\in{\cal U}$, 
$$Q({\bf u},{\bf u})=0,$$
$$Q({\bf u},{\bf v})=1 \text{ if and only if }\langle {\bf u}\rangle ||\langle {\bf v}\rangle.$$
\end{theorem}

\begin{proof}   Since $Q((a,b,c,d),(a,b,c,d)):=(a+b+c+d)^2-(a^2+b^2+c^2+d^2)$, the first claim is obvious.

Recall $P_j$ ($j=1,2,3$) form the vertices of an equilateral triangle with side length 1. Note that since   $|P_i|^2+|P_j|^2-2\Re(\overline
P_iP_j)=|P_i-P_j|^2=1$ if $i\neq j$, the norm squared of an ${\Bbb R}$-linear
combination of the $P_i$ is
$$\begin{aligned}&|xP_1+yP_2+zP_3|^2=(xP_1+yP_2+zP_3)(x\overline P_1+y\overline P_2+
z\overline
P_3)\\&=x^2|P_1|^2+y^2|P_2|^2+z^2|P_3|^2+2xy{\Re}(\overline
P_1P_2)+2xz{\Re}(\overline P_1P_3)+2yz{\Re}(\overline P_2P_3)\\&=
(x+y+z)(x|P_1|^2+y|P_2|^2+z|P_3|^2)-(xy+xz+yz).\end{aligned}$$ 
Hence,
assuming $a+b+c=1=s+t+u$,
$$\begin{aligned}&|(aP_1+bP_2+cP_3)-(sP_1+tP_2+uP_3)|^2
\\&=|(a-s)P_1+(b-t)P_2+(c-u)P_3|^2 \\&=-
[(a-s)(b-t)+(a-s)(c-u)+(b-t)(c-u)]\\&=
(a^2+b^2+c^2+s^2+t^2+u^2-2as-2bt-2cu)/2 .\end{aligned}$$

Let $\langle a,b,c\rangle$ and $\langle s,t,u\rangle$ be two
spheres.  Then the square of the distance between their tangent
points is
$$\dfrac12\left(\dfrac{a^2+b^2+c^2}{(a+b+c)^2}+\dfrac{s^2+t^2+u^2}
{(s+t+u)^2}-\dfrac{2(as+bt+cu)}{(a+b+c)(s+t+u)}\right)$$ and
therefore they are tangent if and only if
$$(s+t+u)\dfrac{a^2+b^2+c^2}{a+b+c}+(a+b+c)\dfrac{s^2+t^2+u^2}{s+t+u}
=2(1+as+bt+cu).\eqno{(5)}$$

Let $d:=-(ab+ac+bc)/(a+b+c)$ and $v:=-(st+su+tu)/(s+t+u)$.  Since
$(a+b+c)^2=a^2+b^2+c^2+2(ab+ac+bc)$, it follows that
$a+b+c=(a^2+b^2+c^2)/(a+b+c)-2d$ and a similar result holds for
$s,t,u$ and $v$.  The left hand side of equation (5) then becomes
$(s+t+u)(a+b+c+2d)+(a+b+c)(s+t+u+2v)=2(a+b+c+d)(s+t+u+v)-2dv$ and
so the spheres are tangent if and only if
$$(a+b+c+d)(s+t+u+v)=1+as+bt+cu+dv.\eqno{(6)}$$
\end{proof}

\noindent{\bf Remark.}  A very direct proof of this theorem is as follows.   Given $\alpha:=x+y\omega$,  $\beta:=u+v\omega$,  $\gamma:=X+Y\omega$, and $\delta:=U+V\omega$ for real numbers $x,y,u,v,X,Y,U,V$,  
$$S_{\alpha,\beta}=\langle a,b,c\rangle$$
where $a: = u^2+v^2-uv+xv-xu-yv, b:= xu-yu+yv,  c:=yu-xv$, and $d: = x^2+y^2-xy+xv-xu-yv$.   Similarly, 
$S_{\gamma,\delta}=\langle A,B,C\rangle$ for similarly defined $A,B,C$, and $D$.    It is then an elementary but tedious calculation to see
that 
$$|\alpha\delta-\beta\gamma|^2=Q((a,b,c,d), (A,B,C,D)).$$

\section{Ford spheres:  tetrahedral case}
\begin{figure}[h] 
   \centering
   \includegraphics[ width=2 in]{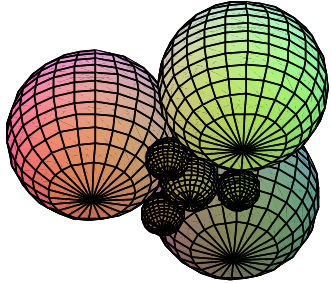} 
   \caption{Ford spheres in tetrahedral arrangement.}
   \label{fig. 3}
\end{figure}

Just as Ford circles ${\cal P}$ were parameterized by ${\Bbb Q}$,  we shall parameterize a certain class of normal spheres ${\cal P}_{\omega}$  by the Eisenstein rationals ${\Bbb Q}(\omega)$.  Further, we shall show that the spheres of ${\cal P}_{\omega}$ can be parametrized barycentrically by solutions of a certain Diophantine equation.  A recursive geometric construction of these spheres is also shown.    In this way, we show an analogue of Theorem 2.4 equating three parameterizations  ${\cal P}_{\omega}, {\cal B}_{\omega},$ and ${\cal G}_{\omega}$.

After defining them appropriately, there are many possible strategies to showing they are all equal.   Ours will be to show 
${\cal G}_{\omega}\subset {\cal P}_{\omega}$,  ${\cal P}_{\omega}\subset{\cal B}_{\omega}$, and, finally, ${\cal B}_{\omega}\subset{\cal G}_{\omega}$.   These spheres have been studied previously by Hellegouch and Rieger in \cite{H1},\cite{H2}, \cite{R1}, and \cite{R2}.   

\vskip .1 in

Recall $\omega:=(-1+i\sqrt 3)/2$ so that  $\omega^3=1$ and $\omega^2=\overline\omega$.  
It is well known that the set of Eisenstein integers ${\Bbb Z}[\omega]:=\{a+b\omega:  a,b\in{\Bbb Z}\}$ is a Euclidean ring (see \cite{IR}).  Its norm is $N(a+b\omega)=|a+b\omega|^2=a^2-ab+b^2$.     As ${\Bbb Z}[\omega]$ is a UFD containing ${\Bbb Z}$,  not every prime in ${\Bbb Z}$ is prime in ${\Bbb Z}[\omega]$.    It turns out though that $\rho\in{\Bbb Z}[\omega]$ is prime if and only if either $N(\rho)$ is the square of a prime congruent to 1 modulo 3 or equals a prime not congruent to 1 modulo 3 (see \cite{IR}).  The \emph{Eisenstein rationals} are the members of the field
$${\Bbb Q}(\omega):=\{r+s\omega:  r,s\in {\Bbb Q}\}=\{\alpha/\beta:  \alpha,\beta\in{\Bbb Z}[\omega], \beta\neq 0\}.$$

${\Bbb Z}[\omega]$ has exactly six units:  $\pm 1, \pm\omega$, and $\pm\overline\omega$.    We say that Eisenstein integers $\alpha,\beta$ are relatively prime (we write $\alpha\perp\beta$) if the only Eisenstein integers dividing both $\alpha$ and $\beta$ are units.    Since ${\Bbb Z}[\omega]$ is a unique factorization domain (UFD), $\alpha\perp\beta$ if and only if  there exist two Eisenstein integers $x,y$ such that $\alpha x+\beta y$ is a unit.  

Recall that for $\alpha,\beta\in{\Bbb C}$, 
$$S_{\alpha,\beta}:=S\left(\frac {\alpha}{\beta}, \frac 1{2|\beta|^2}\right).$$
We define the set of Ford spheres:
$${\cal P}_{\omega}:=\{S_{\alpha,\beta}:  \alpha,\beta\in{\Bbb Z}[\omega],  \alpha\perp\beta\}.$$

By Proposition 3.3, given three mutually tangent spheres, there is a fourth sphere so that these four are mutually tangent.   We say that these four spheres are in a tetrahedral arrangement (since their contact graph is a tetrahedron).     In fact, there are two ways to do this;  we rephrase Proposition 3.3 thus:

\begin{lemma} There are exactly two tetrahedral arrangements of spheres containing three given mutually tangent spheres. \end{lemma}

We now define ${\cal G}_{\omega}$ recursively.    Start with the set ${\cal S}_0:=\{S_{0,1}, S_{1,1}, S_{1+\omega, 1}\}$, perform the following process.   Given ${\cal S}_n$, choose from it three mutually tangent spheres, and add the two as defined in Lemma 4.1 and thus form ${\cal S}_{n+1}$.  In this way, we have a ``chain" of sets ${\cal S}_0\subset {\cal S}_1 \subset ...$.   Let ${\cal G}_{\omega}$ be the union of all elements in all such chains.     Let the \emph{ rank } of a sphere $S$ in ${\cal G}_{\omega}$ denote the length of the shortest (finite) chain containing $S$.   For example,  rank$(S_{1,1})=0$ and rank$(S_{1,1-\omega})=1$.    Hence a sphere of rank $n$ has three ``parents",  each of rank strictly less than $n$,  so that all four are in a tetrahedral arrangement.     Figure 2 illustrates (some of) the spheres of rank at most 2:   
along with the three spheres in ${\cal S}_0$, we include the ``child" $S_{1,1-\omega}$ of those three (the other child, $S_{1,0}$, being a plane parallel to the complex plane, is not shown) as well as children of each of the triples formed by pairs chosen from ${\cal S}_0$ and the sphere $S_{1,1-\omega}$.

\vskip .1 in
  
  \begin{lemma}  ${\cal G}_{\omega}\subset{\cal P}_{\omega}$.  \end{lemma}
  
  \begin{proof}  Note that every element of ${\cal G}_{\omega}$ of rank 0 is in ${\cal P}_{\omega}$.   Suppose that every element of ${\cal G}_{\omega}$ of rank less than $n$ is in ${\cal P}_{\omega}$ and suppose $S\in {\cal G}_{\omega}$ has rank $n$.   Then its parents have smaller rank and are thus in ${\cal P}_{\omega}$.    We will show that $S\in {\cal P}_{\omega}$.   
  
  Let $U:=\{1, 1+\omega, \omega, -1, -1-\omega, -\omega\}$ the set of all six units in ${\Bbb Z}[\omega]$.  Given three mutually tangent spheres $S_{\alpha,\beta}, S_{\gamma,\delta}, S_{x,y}\in {\cal P}$, note that $|\alpha\delta-\beta\gamma|=1$.   Letting $w_1:=x-\alpha$ and $w_2:=y-\beta$,  the tangency of $S_{x,y}$ with the other two spheres implies $$\alpha w_2-\beta w_1\in U,  \delta w_1-\gamma w_2=0.$$
  Hence for some $\rho$,  $w_1=\rho\gamma$ and $w_2=\rho\delta$.    It follows that $\rho( \alpha\delta-\beta\gamma)$ is a unit and thus $\rho$ is too.  
  
  There are exactly two choices of $\sigma\in U$ such that $\sigma-\rho\in U$.    Note that for any such $\sigma$, $S_{\alpha+\sigma\gamma, \beta+\sigma\delta}$ together with the original three spheres forms a tetrahedral arrangement.   Hence $S= S_{\alpha+\sigma\gamma, \beta+\sigma\delta}$ for some unit $\sigma$ and therefore $S\in{\cal P}$.  \end{proof}

Consider now the equation
 $$ (a+b+c+d)^2=a^2+b^2+c^2+d^2.\eqno{(7)}$$
We define ${\Bbb Z}_{\perp}^4$ to be the set of relatively prime integer quadruples and we define a set of Ford spheres ``barycentrically":
$${\cal B}_{\omega}:=\{\langle a,b,c\rangle: (a,b,c,d)\in{\Bbb Z}_{\perp}^4,  (a+b+c+d)^2=a^2+b^2+c^2+d^2, a+b+c>0\}.$$
\begin{lemma} ${\cal P}_{\omega}\subset{\cal B}_{\omega}$.  \end{lemma}

\begin{proof}  
Let $S_{\alpha,\beta}\in{\cal P}_{\omega}$.   Then $\alpha\perp\beta$ where $\alpha=x+y\omega$ and $\beta=u+v\omega$ for some  $x,y,u,v\in{\Bbb Z}$.  Define
$$\begin{aligned}
a&=u^2+v^2-uv+xv-xu-yv,\\
b&=xu-yu+yv,\\
c&=yu-xv,\\
d&=x^2+y^2-xy+xv-xu-yv.
\end{aligned}$$
It is easy to verify 
$$(a+b+c+d)^2=a^2+b^2+c^2+d^2.$$

Then
$$\frac{\alpha}{\beta}=\frac{b+c+c\omega}{a+b+c}, \hskip .2 in \frac1{2|\beta|^2}=\frac1{2(a+b+c)},$$
and so $S_{\alpha,\beta}=\langle a,b,c\rangle$.  

Since  ${\Bbb Z}[\omega]$ is a principal ideal domain and $\alpha\perp\beta$ in it, there exist $\gamma,\delta\in{\Bbb Z}[\omega]$
 such that $|\alpha\delta-\beta\gamma|=1$ and so $S_{\alpha,\beta}||S_{\gamma,\delta}$.   
 
 As above, for some integer quadruple $(A,B,C,D)$ satisfying (7), $S_{\gamma,\delta}=\langle A,B,C \rangle.$
 Hence $\langle a,b,c \rangle||\langle A,B,C \rangle$.   By Theorem 3.4,
 $$(a+b+c+d)(A+B+C+D)-(aA+bB+cC+dD)=1,$$
 and thus $a,b,c,d$ are relatively prime.
 That is, $\langle a,b,c \rangle\in{\cal B}_{\omega}$ and the lemma is shown.\end{proof}

% We shall parameterize a set of spheres (it will turn out to be the set of Ford spheres), by relatively prime integer solutions of (7).

%It is known that the curvatures of four mutually tangent spheres (all tangent to the curvature 0 ``sphere"  ${\Bbb C}$      satisfy (7)  (see \cite{GLMWYiii}, p. ).

%In order to understand the relation between ${\cal B}_{\omega}$ and ${\cal P}_{\omega}$, we first see when two barycentrically spheres are tangent.

%The defining set for ${\cal B}_{\omega}$, namely 
%$${\cal U}: =\{(a,b,c,d)\in{\Bbb Z}_{\perp}^4:   (a+b+c+d)^2=a^2+b^2+c^2+d^2, a+b+c>0\},$$
%is part of the zero set $\{{\bf  u}: Q({\bf u},{\bf u})=0\}$.

\begin{theorem}  ${\cal B}_{\omega}={\cal P}_{\omega}={\cal G}_{\omega}$. \end{theorem}
\begin{proof}  It is enough to show ${\cal B}_{\omega}\subset{\cal G}_{\omega}$.  
Consider the ``generalized slow Euclidean algorithm" (`gSEA'):
$$(a,b,c,d)\longmapsto \begin{cases} (-a,a+b,a+c,a+d) &\text{if $a=\min\{a,b,c,d\}$,}\\
(a+b,-b,b+c,b+d) &\text{if $b=\min\{a,b,c,d\}<a$,}\\
(a+c,b+c,-c,c+d) &\text{if $c=\min\{a,b,c,d\}<a,b$,}\\
(a+d,b+d,c+d,-d) &\text{if $d=\min\{a,b,c,d\}<a,b,c$}.\end{cases}$$
First, notice every solution of  equation (7) is invariant under
the transformation $(a,b,c,d)\mapsto(a+d,b+d,c+d,-d)$.  Also, if
$d=\min\{a,b,c,d\}$ then $(a+d)+(b+d)+(c+d)+(-d)=a+b+c+2d<a+b+c+d$
and, in general, $a+b+c+d$ strictly decreases with every step of
the gSEA. It cannot go below 0 however since $a+b+c\ge a+b+c+d\ge
0$,
$$a+b+c+2d=a+b+c-2(ab+ac+bc)/(a+b+c)=(a^2+b^2+c^2)/(a+b+c)\ge 0.$$
Therefore, if $a+b+c+d>0$, then the gSEA eventually terminates
with $a,b,c,d\ge 0$.   If, say $a,b>0$ at that stage, then
$0=ab+ac+ad+bc+bd+cd\ge ab>0$ -- a contradiction;  therefore, at
most one of $a,b,c,d$ is non-zero.   Since the gcd is preserved by
the gSEA, the gSEA eventually terminates at one of
$(g,0,0,0),(0,g,0,0),(0,0,g,0),(0,0,0,g)$ where
$g:=\gcd(a,b,c,d)$.
Every ${\bf u}\in {\cal U}$ must end in a basis vector $(1,0,0,0), (0,1,0,0), (0,0,1,0)$ or $(0,0,0,1)$ and reversing
the gSEA in each of the other three basis vectors gives rise to three new 
parents since tangency is preserved by the gSEA.  This is the ``parent algorithm".

Given a relatively prime integer solution of (7), let its \emph{rank} be the number of steps taken by the gSEA.   
For example,
$$\begin{aligned}(12,12,3,-8)&\mapsto(4,4,-5,8)\mapsto(-1,-1,5,3)\mapsto(1,-2,4,2)\\&
\mapsto(-1,2,2,0)\mapsto(1,1,1,-1)\mapsto(0,0,0,1).\end{aligned}$$
We code the steps:  $[4,3,1,2,1,4]$ and the rank of $(12,12,3,-8)$ is 6.  Reversing these (and noting that each step is idempotent), and applying to the standard basis vectors, we get:
$$\begin{aligned}
&(0,0,1,0) \mapsto (0,0,1,0) \mapsto ... \mapsto (2,2,0,-1)\\
&(0,1,0,0) \mapsto (0,1,0,0) \mapsto ... \mapsto (5,6,2,-4)\\
&(1,0,0,0) \mapsto (1,0,0,0) \mapsto ... \mapsto (6,5,2,-4).\end{aligned}$$
Note that the first step always leaves the vector fixed and so the parents $(2,2,0,-1)$,  $(5,6,2,-4)$, $(6,5,2,-4)$ of $(12,12,3, -8)$ have smaller rank.   

Note that every element of ${\cal B}_{\omega}$ of rank 0 is in ${\cal G}_{\omega}$. 
Suppose that every element of ${\cal B}_{\omega}$ of rank less than $n$ is in ${\cal G}_{\omega}$ and suppose $S\in {\cal B}_{\omega}$ has rank $n$.   Then its parents have smaller rank and are thus in ${\cal G}_{\omega}$.    It is easy to check that  $S\in {\cal G}_{\omega}$.     By induction on rank, the theorem is shown.  \end{proof}

Three corollaries immediately follow.

\begin{corollary}  The integer solutions of $$(a+b+c+d)^2=a^2+b^2+c^2+d^2$$ are parameterized by 
$$\begin{aligned}
a&=u^2+v^2-uv+xv-xu-yv,\\
b&=xu-yu+yv,\\
c&=yu-xv,\\
d&=x^2+y^2-xy+xv-xu-yv.
\end{aligned}$$\end{corollary}

The following Corollary solves a recent Monthly problem \cite{Nprob}
\begin{corollary}   If $a,b,c,d$ are relatively prime and satisfy $(a+b+c+d)^2=a^2+b^2+c^2+d^2$, then $|a+b+c|=m^2+mn+n^2$ for some integers $m,n$. \end{corollary}
\begin{proof} Without loss of generality, assume $a+b+c>0$.  By hypothesis,  $\langle a,b,c \rangle$ is a Ford sphere with radius $1/2(a+b+c)$.   By Theorem 4.4 , this sphere is also of the form 
$S_{\alpha, \beta}$ with radius $1/2|\beta|^2$ for some Eisenstein integers $\alpha, \beta$.   Hence $|a+b+c|=|\beta|^2=m^2+mn+n^2$ for some integers $m,n$. \end{proof}

Theorems 3.3 and 4.4 provide a useful way to construct Ford spheres recursively.

\begin{corollary}{\bf (Tetrahedral Rule).} For ${\bf a}, {\bf b}, {\bf c}, {\bf d}\in{\cal U}$,  if $\langle{\bf a}\rangle, \langle{\bf b}\rangle, \langle{\bf c}\rangle,\langle{\bf d}\rangle$ are mutually tangent, then so are $\langle{\bf a}\rangle, \langle{\bf b}\rangle, \langle{\bf c}\rangle,\langle{\bf a}+{\bf b}+{\bf c}-{\bf d}\rangle$.  \end{corollary}
\begin{proof}  $Q({\bf a},{\bf a}+{\bf b}+{\bf c}-{\bf d})=Q({\bf a},{\bf a})+Q({\bf a},{\bf b})+Q({\bf a},{\bf c})-Q({\bf a},{\bf d})=0+1+1-1=1.$ \end{proof}

The reason we call it the ``tetrahedral rule" should be clear from Figure 3 which represents several spheres as triples of integers obeying the tetrahedral rule.   It is clear that this diagram indicates a tesselation of (part of) ${\Bbb R}^3$ by tetrahedra such that every edge is shared by exactly six tetrahedra.   The group of symmetries here is then a reflection group with four generators (Coxeter diagram a tetrahedron with each edge weight three ($(ab)^3=e$, etc.).    This group is represented as a group of matrices $M_k$ defined by
$$(M_k)_{ij}=\delta_{ij}+\delta_{ik}-3\delta_{ik}\delta_{jk}.$$  This group seems related to the Eisenstein-Picard modular group \cite{FP}.

\begin{figure}[htbp] 
   \centering
   \includegraphics[ width=2 in]{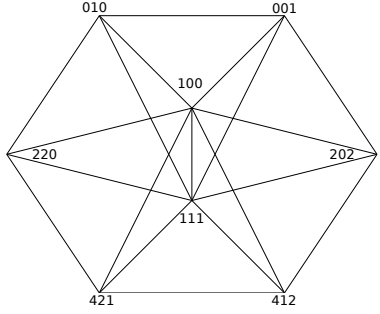} 
   \caption{Tetrahedral rule.}
   \label{fig. 2}
\end{figure}

%Picture Ref:  Bl-128 and SpherePackingPix.mws,  4/26/08.

An interesting side question is ``when does the gSEA eventually
repeat?". The SEA of Section 2, starting with $(1,x)$ eventually repeats if and only if 
the continued fraction for $x$ eventually repeats (via the bijection $(a,b)\leftrightarrow b/a$).  

It
turns out that this is contained in the gSEA case. Note that for
any $x\in{\Bbb R}$, $(1,x,x^2,-x)$ is a solution of equation (7). We
shall show that the gSEA applied to this vector eventually repeats
if and only if $x$ is a quadratic surd.  To help, we say that two
vectors $A,B$ are equivalent ($A\equiv B$) if a permutation of one
is a scalar multiple of the other (e.g.
$(4,1,3,2)\equiv(2,4,6,8)$).

Suppose first that $x>1$.  Then
$$(1,x,x^2,-x)\mapsto(1-x,0,x^2-x,x)\mapsto(x-1,1-x,(x-1)^2,1)\equiv(1,y,y^2,-y)$$
where $y=x-1$.

On the other hand, if $0<x<1$, then
$$\begin{aligned} (1,x,x^2,-x)&\mapsto((1-x,0,x^2-x,x)\\&\mapsto((1-x)^2,-x(1-x),x(1-x),x^2)\equiv(1,y,y^2,-y)\end{aligned}$$
where $y=x/(1-x)$.  Hence the double-stepped gSEA is, modulo
$\equiv$, equivalent to iteration of
$$x\longmapsto f(x):=\begin{cases} x-1 &\text{if $x>1$,}\\
\frac x{1-x} &\text{if $0<x<1$,}\end{cases}$$ This is closely
related to continued fractions: if $x$ has continued fraction
expansion $[a_0,a_1,a_2,a_3,...]$, then it is easy to verify that
$$\begin{aligned} f_{a_0}(x)&=1/[a_1,a_2,a_3,...]\\
f_{a_0+a_1}(x)&=[a_2,a_3,a_4,...]\\
f_{a_0+a_1+a_2}(x)&=1/[a_3,a_4,a_5,...]\\
\text{ etc. }\end{aligned}$$ Hence the gSEA beginning with
$(1,x,x^2,-x)$ eventually repeats if and only if $x$ is a
quadratic surd.

%%%%%%%%%%%%%%%%%%%%%%%

\section{Ford spheres:  octahedral case}

\begin{figure}[h] 
   \centering
   \includegraphics[ width=2 in]{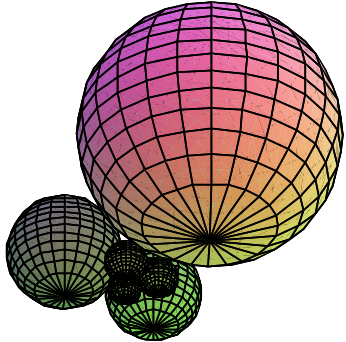} 
   \caption{Octahedral Spherical Array}
   \label{fig. 7}
\end{figure}

In Section 4, we introduced an array ${\cal P}_{\omega}$ of spheres parameterized by the Eisenstein rationals ${\Bbb Q}(\omega)$,  a field whose ring of integers is a Euclidean domain.   In this Section, we introduce an array ${\cal P}_i$ of normal spheres parameterized by the Gaussian rationals ${\Bbb Q}(i)$,  another field whose ring of integers is a Euclidean domain.   In Section 4, we gave a recursive geometric procedure for constructing ${\cal P}_{\omega}$ based on repeatedly, for any three mutually tangent spheres,  adding another so that all four have contact graph a tetrahedron.    In this Section, we shall give a recursive geometric procedure for constructing ${\cal P}_i$ that repeatedly, for any three mutually tangent spheres,  adds three others, so that all six have contact graph an octahedron.   Figure 4 shows three mutually tangent spheres together with their three children so as to form an octahedral array.  In Section 4, we also gave a barycentric parameterization of the spheres in ${\cal P}_{\omega}$ by expressing them in the form $\langle a,b,c\rangle$ for each relatively prime integer solution $(a,b,c,d)$ of $(a+b+c+d)^2=a^2+b^2+c^2+d^2$.  By replacing this equation by 
$$(a+b+c+d)^2=2(a^2+b^2+c^2+d^2),\eqno{(8)}$$
we get a new array of spheres, which under the Poincar\'e extension $M$ of a certain M\"obius transformation, is the same as ${\cal P}_i$.

We shall first define ${\cal P}_i$ and, by using the (slow) Euclidean algorithm for Gaussian integers, show that it is maximal in the sense that no normal sphere can be included or enlarged.     Next we show that octahedral arrays of spheres exist and that there are exactly two such arrays containing a given triple of mutually tangent spheres.   This allows for a way of defining a set of spheres ${\cal G}_i$ geometrically in terms of octahedra just as we did earlier for tetrahedra.    We then show that ${\cal P}_i\subset {\cal G}_i$ (and so they are equal by the maximality of ${\cal P}_i$).     Next we define a set of spheres ${\cal B}_i$, parameterized by solutions of (8), and show that ${\cal P}_i\subset M({\cal B}_i)$ where $M$ is the Poincar\'e extension of a certain M\"obius transformation.    We may thus conclude that ${\cal P}_i={\cal G}_i=M({\cal B}_i)$.   

These spheres,  parameterized as in the definition of ${\cal P}_i$, have been studied previously by Pickover \cite{Pick}.  

 \vskip .1 in

It is well known that the set of Gaussian integers ${\Bbb Z}[i]:=\{a+ib:  a,b\in{\Bbb Z}\}$ is a Euclidean ring (see \cite{IR}).  Its norm is $N(a+ib)=|a+ib|^2=a^2+b^2$.     As ${\Bbb Z}[i]$ is a UFD containing ${\Bbb Z}$,  not every prime in ${\Bbb Z}$ is prime in ${\Bbb Z}[i]$.     The \emph{Gaussian rationals}, are the members of the field
$${\Bbb Q}(i):=\{r+is:  r,s\in {\Bbb Q}\}=\{\alpha/\beta:  \alpha,\beta\in{\Bbb Z}[i], \beta\neq 0\}.$$

${\Bbb Z}[i]$ has four units: $1,-1,i$, and $-i$.  We say that Gaussian integers $\alpha,\beta$ are relatively prime (we write $\alpha\perp\beta$) if the only Gaussian integers dividing both $\alpha$ and $\beta$ are units.    Since ${\Bbb Z}[i]$ is a unique factorization domain (UFD), $\alpha\perp\beta$ if and only if  there exist two Gaussian integers $x,y$ such that $\alpha x+\beta y$ is a unit.  

Recall that for $\alpha,\beta\in{\Bbb C}$, 
$$S_{\alpha,\beta}:=S\left(\frac {\alpha}{\beta}, \frac 1{2|\beta|^2}\right).$$
For this Section, we define the set of Ford spheres as
$${\cal P}_i:=\{S_{\alpha,\beta}:  \alpha,\beta\in{\Bbb Z}[i],  \alpha\perp\beta\}.$$

For any $S_{\alpha,\beta}\in {\cal P}_i$,  $\alpha\perp\beta$ and thus there exist Gaussian integers $\gamma, \delta$ such that $|\alpha\delta-\beta\gamma|= 1$.   Hence $S_{\alpha,\beta}||S_{\gamma+\rho\gamma, \delta+\rho\delta}$ for any $\rho\in{\Bbb Z}_i$.   For any two spheres $S_{\alpha,\beta}, S_{\gamma, \delta}\in {\cal P}_i$,  since $\alpha\delta-\beta\gamma$ is a Gaussian integer,  $|\alpha\delta-\beta\gamma|\ge 1$ and so $S_{\alpha,\beta}$ and  $S_{\gamma, \delta}$ have disjoint interiors.  Hence a given Ford sphere is tangent to infinitely many others but does not intersect the interior of any other.

\begin{lemma}  For all $z\in{\Bbb C}$, 
$$\inf\{|\beta z-\alpha|: \alpha,\beta\in{\Bbb Z}[i]\}=0.$$ \end{lemma}
\begin{proof}    We may define a ``floor function" for ${\Bbb C}$:  
$$\lfloor x+iy\rfloor:=\lfloor x \rfloor+ i\lfloor y \rfloor$$
and a corresponding ``fractional part" function
$$\{x+iy\}:=\{x\}+i\{y\}.$$
Given positive integer $N$, divide the square $R:=[0,1)\times [0,1)$ into $N^2$ disjoint congruent squares.  
For  $z\not\in{\Bbb Q}(i)$, the numbers in $\{(m+in)z: m,n=1,..., N+1\}$ are all distinct and so, by the pigeonhole principle, there exists $\alpha,\beta\in{\Bbb Z}[i]$ such that $\{\alpha z\}$ and $\{\beta z\}$ are in the same small sub-square of $R$.   
There are thus $\gamma,\delta\in{\Bbb Z}[i]$ such that
$$|(\alpha-\beta)z+(\gamma-\delta)|=|\{\alpha z\}-\{\beta z\}|\leq\frac{\sqrt 2}N.$$
Since $N$ was arbitrary, the result follows for any  $z\not\in{\Bbb Q}(i)$.   
The result obviously holds for $z\in{\Bbb Q}(i)$. \end{proof}

\begin{lemma}  ${\cal P}_i$ is maximal. \end{lemma}
\begin{proof}  
Note that any two Ford spheres $S_{\alpha,\beta},S_{\gamma,\delta}$ have disjoint interiors (since $|\alpha\delta-\beta\gamma|\ge 1$).  Furthermore, since $\alpha\perp\beta$, there exist $\gamma,\delta$ such that $|\alpha\delta-\beta\gamma|= 1$ and so  $S_{\alpha,\beta}||S_{\gamma,\delta}$.   Hence no Ford sphere can be enlarged.   

     Suppose there exist $z,r$ so that  for all $a,b$, $S(z,r)^{\circ}\cap S_{\alpha,\beta}= \emptyset$.  Obviously $z\not\in{\Bbb Q}(i)$  and thus, by Lemma 5.1, the set $\{\beta z-\alpha: \alpha\in{\Bbb Z}[\omega], \beta\in{\Bbb Z}[\omega]^+ \}$ is dense in ${\Bbb C}$  and so, for any fixed $r$, there exist $\alpha,\beta$ such that 
$$\left|z-\frac {\alpha}{\beta}\right|<\frac{\sqrt{2r}}{|\beta|}.$$
This implies $S(z,r)^{\circ}\cap S_{\alpha,\beta}^{\circ}\neq\emptyset$, a contradiction.  \end{proof}

We now define a set of spheres ${\cal R}_i$ recursively.    Start with the set ${\cal S}_0:=\{S_{0,1}, S_{1,0}\}$, perform the following process.   Given ${\cal S}_n$, create ${\cal S}_{n+1}$ by choosing two tangent spheres, say $S_{\alpha,\beta}, S_{\gamma, \delta}$ in ${\cal S}_n$ and add the sphere $S_{\alpha+\rho\gamma, \beta+\rho\delta}$ for one of the units $\rho\in\{1,-1,i,-i\}$.   In this way, 
we have a ``chain" of sets ${\cal S}_0\subset {\cal S}_1 \subset ...$.   Let ${\cal R}_i$ be the union of all sets in all such chains.   

\begin{lemma}   ${\cal P}_i={\cal R}_i$. \end{lemma}

\begin{proof}  By Lemma 5.2, It is enough to show that ${\cal P}_i\subset {\cal R}_i$.  
We shall do this by utilizing the ``slow Euclidean algorithm" for the Gaussian integers.

Given Gaussian integers $\alpha$ and $\beta$, one can choose a unit $\rho$ so that the angle between $\alpha$ and $\rho\beta$, as vectors, is in $[-\pi/4, \pi/4]$.  
Therefore, if $0<|\beta|\le|\alpha|$, then there is a unit $\rho$ such that $|\alpha-\rho\beta|<|\alpha|$.   
We now define $SEA_i$, the slow Euclidean algorithm for ${\Bbb Z}[i]$:  
\[
[\alpha,\beta] \longmapsto
\begin{cases}
   [\alpha-\rho\beta, \beta], &\text { if $0<|\beta|<|\alpha|$ and $\rho\in{\cal U}$ with $|\alpha-\rho\beta|<|\alpha|$ }\\
 [\alpha, \beta-\rho\alpha], &\text { if $0<|\alpha|\le|\beta|$ and $\rho\in{\cal U}$ with $|\beta-\rho\alpha|<|\beta|$ }\\
 \text {stop}, &\text { if $\alpha\beta=0$.}
 \end{cases}
 \]
 Note that $|\alpha|^2+|\beta|^2$ is decreasing, positive, integer valued so the algorithm must stop for any two Gaussian integers $\alpha,\beta$.
 Since the ``greatest" common divisors are preserved at each step, the algorithm must end at either $(0,\rho)$ or $(\rho,0)$ for some $\rho$.   Further, if $\alpha\perp\beta$, then $\rho$ must be a unit.
 We say that the ``rank" of a pair $[\alpha,\beta]$ is the number of steps taken by $SEA_i$ to end.

As with the parent algorithm for Ford circles,  run $SEA_i$ until the penultimate state $[\rho_1,\rho_2]$ (where $\rho_1, \rho_2\in{\cal U}$) and then reverse the steps on $[\rho_1,0]$ and $[0,\rho_2]$.  The resulting pairs, call them $[\alpha_1, \beta_1]$ and $[\alpha_2, \beta_2]$ are ``parents" of $[\alpha,\beta]$ in the sense that each has lower rank than $[\alpha,\beta]$ and the two must sum to  $[\alpha,\beta]$.   

Every sphere of rank 0 is in ${\cal R}_i$.   Suppose that every sphere of rank less than $n$ is in ${\cal R}_i$ and suppose $S_{\alpha,\beta}\in{\cal P}_i$.   Then its parents have lower rank and so are in ${\cal R}_i$.   By the definition of ${\cal R}_i$,  $S_{\alpha,\beta}\in{\cal R}_i$.  By (strong) induction, the lemma is shown.  
 \end{proof}

Given three non-colinear points $P_1,P_2,P_3$, we consider them on the boundary of a hyperbolic disk.    There is then a unique geodesic triangle with  vertices $P_1,P_2,P_3$ and, a unique circle inscribed in that triangle.   This gives rise to three unique points $Q_1,Q_2,Q_3$ such that the six points $P_1,P_2, P_3, Q_1,Q_2,Q_3$ form an octahedron (that we call a \emph{M\"obius octahedron} generated by $P_1,P_2,P_3$ since a M\"obius transformation changes it to another such octahedron) -- see Figure 4.

\begin{figure}[htbp] 
   \centering
   \includegraphics[ width=1.6 in]{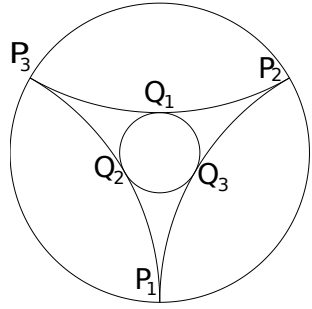} 
   \caption{M\"obius Octahedron}
   \label{fig. 4}
\end{figure}

The \emph{contact graph} of a finite collection $V$ of normal spheres is the graph with vertex set $V$ where two vertices share an edge if the two spheres are tangent.   For example, three mutually tangent spheres have tangency graph a triangle.    

\begin{theorem} The six points of a M\"obius octahedron are the tangent points of six spheres with contact graph an octahedron.    \end{theorem}  

\begin{proof} Recall the cross ratio:  $$[z,q,r,s]:=\dfrac{(z-q)(r-s)}{(z-s)(r-q)}.$$
It is well-known that it is invariant under M\"obius transformations:  
$$[m(z),m(q),m(r),m(s)]=[z,q,r,s]$$ (see Needham \cite{Ne}).
Let $A,B,C,D,E,F$ be vertices of a M\"obius octahedron where $A,B,C$ form a triangle and $D,E,F$ form a triangle.   Let $(AB)$ denote $|A-B|$, etc.  It is easy to verify that 
$$|[F,C,B,D][A,B,D,E][A,C,F,E]|=\left| \dfrac{(AB)(AC)(DE)(EF)}{(AE)^2(BC)(DF)}\right|$$
and so the right side is invariant under M\"obius transformations.  For the special case where $A=0, B=(i+1)/2, C=1,D=1+i,E=\infty$, and $F=i$, the right side is 1 and since every M\"obius octahedron is the image of the special one under a M\"obius transformation,
$$(AE)^2=\dfrac{(AB)(AC)(ED)(EF)}{(BC)(BF)}.\eqno{(9)}$$
By Proposition 3.1, there are three mutually tangent normal spheres tangent to the plane at $A,B,C$ and another three tangent at $D,E,F$.  By (9) and Proposition 4.1, $|A-E|^2=4r_Ar_E$ and so the spheres at $A$ and $E$ are tangent.   Similar arguments for pairs $AF, BD,BF, CD,$ and $CE$ show that the six spheres have an octahedral tangency graph.
\end{proof}

``Extended" cross-ratios like those in the proof of Theorem 5.4 make an appearance in the paper \cite{KS}.

\begin{lemma} Given three mutually tangent spheres, there are exactly two ways to choose another set of three mutually tangent spheres so that those six have contact graph an octahedron.   Further, every collection of six normal spheres with contact graph an octahedron are tangent to the plane at six points in some M\"obius octahedron. 

\end{lemma}

\begin{proof}  
 Given three mutually tangent spheres, let $z_1,z_2,z_3$ denote where they are tangent to ${\Bbb C}$.  
The M\"obius transformation $m(z):=[z,z_1,z_2,z_3]$ defined in terms of the cross ratio takes $z_1,z_2,z_3$ to $0,1,\infty$ respectively (see Figure 6).   
 Its Poincar\'e extension takes the three spheres to $S_{0,1}, S_{1,1}, S_{1,0}$.   It's easy to verify that these three spheres together with the spheres $S_{\rho, \rho+1}, S_{\rho, 1}, S_{1+\rho, 1}$ (where $\rho=i$ or $-i$) forms an octahedral arrangement and so the reverse M\"obius transformation $m^{-1}$ extends either one to an octahedral arrangement.  
  There are no more than these two since the points $0,1,\infty$ form a triangle in a M\"obius octahedra if and only if that octahedron has vertices $\{0,1,\infty, i, 1+i, (1+i)/2\}$ or $\{0,1,\infty, -i, 1-i, (1-i)/2\}$.
  
  Given six normal spheres with contact graph an octahedron,  choose three mutually tangent spheres that are tangent to ${\Bbb C}$ at, say, $z_1, z_2, z_3$.   The image of all six spheres under the Poincar\'e extension of $[\cdot, z_1,z_2,z_3]$ is then, by the argument above, vertices of a M\"obius octahedron (see Figure 6) and the result follows.  \end{proof}
  
  \begin{figure}[htbp] 
   \centering
   \includegraphics[ width=1.8
   in]{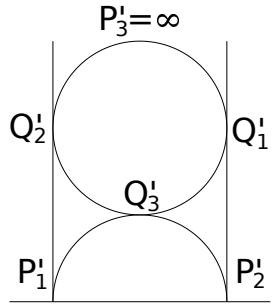} 
   \caption{Image of Fig. 4 under $[\cdot, P_1,P_2,P_3]$}
   \label{fig. 6}
\end{figure}

% Now, suppose there are six normal spheres with octahedral tangency graph.   Let $P_j$  $(j=1,...,6)$ be tangent points and let $S_j$ and $r_j$  be the respective spheres and radii.  Let $(jk):=|P_j-P_k|$.   Then, by Proposition 4.1,  $$2r_1=\dfrac{(12)(16)}{(26)}=\dfrac{(16)(15)}{(65)}$$ and so  $$|[P_1,P_2,P_6,P_5]|=\dfrac{(12)(65)}{(15)(62)}=1.$$  In general, every 4-cycle $P_j,P_k,P_m,P_n$ satisfies $[P_j,P_k,P_m,P_n]=1$.   

%Now, let $m(z):=[z,P_1,P_2,P_3]$ and define $P_i':=m(P_j)$.    Then $P_1':=0, P_2'=1$, and $P_3'=\infty$.    Let $r_j':=|P_j'-P-m'||P_j'-P-n'|/2|P_m'-P_n'|$ where $|P_j-P-m||P_m-P-n|/2|P_m-P_n|$.   This is well-defined since the cross ratio is invariant under M\"obius transformations.   Hence, the spheres $S_j'$ tangent at $P_j'$ with radius $r_j'$ form an octahedral arrangement of normal spheres.   Since $S_3'$ coincides with the plane parallel to but either one unit above or below the complex plane, $r_1'=r_2'=r_4'=r_5'=1/2$ and so the points $P_1',P_2',P_4',P_5'$ are vertices of a rhombus.   Since $S_6'$ is tangent to each of  $S_1',S_2', S_4', S_5'$, it follows that the points $P_1',P_2',P_4',P_5'$  are the vertices of a square and $P_6'$ is the center of that square.     The points $P_j'$ are thus the vertices of a M\"obius octahedron and therefore so are the points $P_j$.  

We may thus define a new type of Ford sphere geometrically.
We now define ${\cal G}_i$ recursively.    Start with the set ${\cal S}_0:=\{S_{0,1}, S_{1,1}, S_{1,0}\}$, perform the following process.   Given ${\cal S}_n$, choose three spheres, and add the three as defined in Lemma 5.1 and thus form ${\cal S}_{n+1}$.  In this way, we have a ``chain" of sets ${\cal S}_0\subset {\cal S}_1 \subset ...$.   Let ${\cal G}_i$ be the union of all sets in all such chains.

\begin{lemma}  ${\cal R}_i\subset {\cal G}_i$.  \end{lemma}
\begin{proof}
We define, for $U:=(\alpha, \beta)\in{\Bbb Z}[i]^2$,  $S[U]:= S_{\alpha,\beta}$.   Note that for any unit $\rho$,  $S[\rho U]=S[U]$.   

 It is enough to show that if $S[U], S[V]$ are two tangent spheres in ${\cal G}_i$ and $\rho$ is a unit, then $S[U+\rho V]\in{\cal G}_i$.     Suppose  $S[U], S[V]$ are two tangent spheres in ${\cal G}_i$.   The two spheres are then part of an octahedral arrangement and so part of a triangle which, without loss of generality, is of the form $S[U], S[U+V], S[V]$.  
It is easy to verify that for $\rho=\pm i$,   $X_{\rho}:=\{S[U], S[U+V], S[V], S[U+\rho(U+V)], S[U+\rho V], S[U+V+\rho V]\}$  forms an octahedral arrangement of spheres and thus, by Lemma 5.5, these are the only octahedral extensions of $S[U], S[U+V], S[V]$.    In this way, we see that $S[U+\rho V]\in{\cal G}_i$ for $\rho =1, i, -i$.   The triangle of spheres $S[U], S[U+iV], S[V]$ is part of two octahedra, one of which contains $S[U-V]$ and so the lemma is shown. 
\end{proof}

We now reconsider equation (8):
 $$ (a+b+c+d)^2=2(a^2+b^2+c^2+d^2).$$
 A solution of (8) is known as a \emph{Descartes quadruple} and it was known (by Descartes) that
 that the curvatures of four mutually tangent circles satisfy (8)  (see \cite{GLMWYnt}).
We shall parameterize the sets set of spheres ${\cal P}_i$ by relatively prime integer solutions of (8).

As before, we define ${\Bbb Z}_{\perp}^n$ to be the set of relatively prime integer $n$-tuples.
Let ${\cal D}$ denote the set of relatively prime solutions of (8):
$${\cal D}:=\{(a,b,c,d)\in{\Bbb Z}_{\perp}^4:  (a+b+c+d)^2=2(a^2+b^2+c^2+d^2)\}.$$
A related set, the set of ``Descartes triples", can be defined:
$${\cal S}:=\{(a,b,c)\in{\Bbb Z}_{\perp}^3:  \sqrt{ab+ac+bc}\in{\Bbb Z}\}.$$
We call them Descartes triples since it
is easy to see that if $(a,b,c)\in {\cal S}$, then $(a,b,c,a+b+c\pm 2\sqrt{ab+ac+bc})\in {\cal D}$ and, conversely, if
$(a,b,c,d)\in {\cal D}$ then $(a,b,c)\in {\cal S}.$

Let $$M(z):=iz/((1-z)\omega) \text{ and } M^{-1}(z):=\omega z/(\omega z+i).$$
Then $M$ is the M\"obius transformation taking the points $0,1, 1+\omega$ (the vertices of an equilateral triangle of side length 1) to the points $0, \infty, i$ respectively.  
Using the notation of Section 3, 

\begin{lemma} Given $\alpha, \beta\in{\Bbb C}$, let $a:=|\beta|^2+\Im(\overline\alpha\beta), b:=|\alpha|^2+\Im(\overline\alpha\beta),  c:=-\Im(\overline\alpha\beta)$, and
$m=\Re(\overline\alpha\beta)$.   Then
$$M^{-1}(S_{\alpha, \beta})=\langle a+m/\sqrt 3, b+m/\sqrt 3, c+m/\sqrt 3\rangle$$
and $|m|=\sqrt{ab+ac+bc}$. 
\end{lemma}

\begin{proof}  By Section 3, 
$$M^{-1}(S_{\alpha,\beta})=S_{\omega\alpha,\omega\alpha+i\beta}=S\left(\frac{\omega\alpha}{\omega\alpha+i\beta}, \frac1{2|\omega\alpha+i\beta|^2}\right).$$
For this sphere to agree with $\langle A,B,C\rangle:=S((B+C(1+\omega))/(A+B+C), 1/2(A+B+C)),$
we must have
$$A+B+C=|\omega\alpha+i\beta|^2=|\alpha|^2+|\beta|^2+\sqrt 3 \Re(\overline\alpha\beta)+\Im(\overline\alpha\beta)$$
and 
$$\begin{aligned} B+C(1+\omega)&=\omega\alpha(\overline{\omega\alpha+i\beta})=|\alpha|^2-i\omega\alpha\overline\beta\\
&=|\alpha|^2+\frac12\sqrt 3\Re(\overline\alpha\beta)+\frac12\Im(\overline\alpha\beta)
+i\left( \frac12\Re(\overline\alpha\beta)-\frac12\sqrt 3\Im(\overline\alpha\beta)\right). \end{aligned}$$
Solving for $A,B,C$, we find 
$$\begin{aligned} &C=\Re(\overline\alpha\beta)/\sqrt 3-\Im(\overline\alpha\beta),\\
&B=|\alpha|^2+\Re(\overline\alpha\beta)/\sqrt 3+\Im(\overline\alpha\beta),\\
&A=|\beta|^2+\Re(\overline\alpha\beta)/\sqrt 3+\Im(\overline\alpha\beta),\end{aligned}$$
 and the first part of the lemma follows.
 
 With $a,b,c$ defined above, it is easy to verify that 
 $$ab+ac+bc=|\alpha|^2|\beta|^2-\Im(\overline\alpha\beta)^2=\Re(\overline\alpha\beta)^2=m^2.$$ \end{proof}

We now define a type of Ford sphere barycentrically:  let $${\cal B}_i:=\{\langle a+m/\sqrt 3, b+m/\sqrt 3, c+m/\sqrt 3\rangle:  (a,b,c)\in {\cal S}, m=\pm\sqrt{ab+ac+bc}\}.$$

\begin{lemma} ${\cal P}_i\subset M({\cal B}_i)$.  \end{lemma}

\begin{proof}  Define 
$$V(a,b,c):=(a+m/\sqrt 3, b+m/\sqrt 3, c+m/\sqrt 3, -2m/\sqrt 3)$$
where $m=\sqrt{ab+ac+bc}$.  
It is easy to check that $V(a,b,c)$ satisfies equation (7).   Further, with $Q$ the bilinear form defined in Section 3,
$$\begin{aligned} &Q(V(a,b,c), V(A,B,C))\\&=aB+aC+bA+bC+cA+cB-2\sqrt{ab+ac+bc}\sqrt{AB+AC+BC}.\end{aligned}\eqno{(10)}$$

Let $S_{\alpha,\beta}\in {\cal P}_i$.   Then $\alpha\perp\beta\in {\Bbb Z}[i]$ and so $|\alpha|^2, |\beta|^2, \Re(\overline\alpha\beta), \Im(\overline\alpha\beta) \in {\Bbb Z}$.  
By Lemma 5.7, 
$$S_{\alpha, \beta}=M(\langle a+m/\sqrt 3, b+m/\sqrt 3, c+m/\sqrt 3\rangle)$$
for integers $a,b,c$.   
Since $\alpha\perp\beta$  there exist Gaussian integers $\gamma,\delta$ so that $|\alpha\delta-\beta\gamma|=1$ and also $\gamma\perp\delta$.  
Then $S_{\gamma,\delta}=M(\langle A+n/\sqrt 3, B+n/\sqrt 3, C+n/\sqrt 3\rangle)$ for integers $A,B,C,n$.
By (10), $aB+aC+bA+bC+cA+cB-2mn=1$ and thus $(a,b,c)\in {\cal S}$ and thus $S_{\alpha,\beta}\in M({\cal B}_i)$.

\end{proof}

Together, the lemmas of this Section imply

\begin{theorem}  ${\cal G}_i = {\cal P}_i = M({\cal B}_i)$. \end{theorem}

\begin{corollary}  For a Descartes quadruple $(a,b,c,d)$, 
$$a+b+c=|\gamma|^2+|\delta|^2$$
 for some $\gamma,\delta\in {\Bbb Z}[\omega]$, and
 $$ a+b=m^2+n^2$$
 for some $m,n\in{\Bbb Z}$.  \end{corollary}
 
 \begin{proof} 
 Using the proof above,
 $$\begin{aligned} &a+b+c=|\alpha|^2+|\beta|^2+\Im(\overline\alpha\beta)\\&=(a_1^2+a_1b_2+b_2^2)+(a_2^2-a_2b_1+b_1^2)=|a_1-b_2\omega|^2+|a_2+b_1\omega|^2.\end{aligned}$$
 
 Given Descartes triple $(a,b,c)$,  let $S(z,r)=M(\langle a+m/\sqrt 3, b+m/\sqrt 3, c+m/\sqrt 3\rangle)$ where $M$ is as above and $m^2=ab+ac+bc$.  Then $z$ is a Gaussian rational and there exist relatively prime Gaussian integers $\alpha, \beta$ with $z=\alpha/\beta$.    Then $a=|\beta|^2+\Im(\overline\alpha\beta)$,  $b=|\alpha|^2+\Im(\overline\alpha\beta)$, and so 
 $$a+b=|\alpha|^2+|\beta|^2+2\Im(\overline\alpha\beta)=(a_1+b_2)^2+(a_2-b_1)^2.$$ \end{proof}

%%%%Comments

The close relation between continued fractions and Ford circles has long been noted (see \cite{Se} for a thorough study of the relation between continued fractions and the dual of Ford circles).    Higher dimensional analogues have also been studied by many authors (generally by extending the compact interval to a triangular region).   Another option, not previously studied, is to extend the theory of continued fractions to the ``Sierpinski gasket",  a natural generalization of the closed interval.   In attempting this, one may try to generalize Ford circles to spheres on the Sierpinski gasket.  

A compact interval can be represented 
$$[P_1,P_2]=\left\{\sum_j 2^{-j}P_{f(j)}: f\in\{1,2\}^{\Bbb N}\right\}.$$
This generalizes easily to any dimension;   in dimension 3, we have the ``Sierpinski gasket"
$${\bf SG}:=\left\{\sum_j 2^{-j}P_{f(j)}: f\in\{1,2,3\}^{\Bbb N}\right\}$$
where $P_1, P_2, P_3$ are complex numbers forming the vertices of an equilateral triangle.  

\begin{figure}[htbp] 
   \centering{
   \includegraphics[ width=2.2 in]{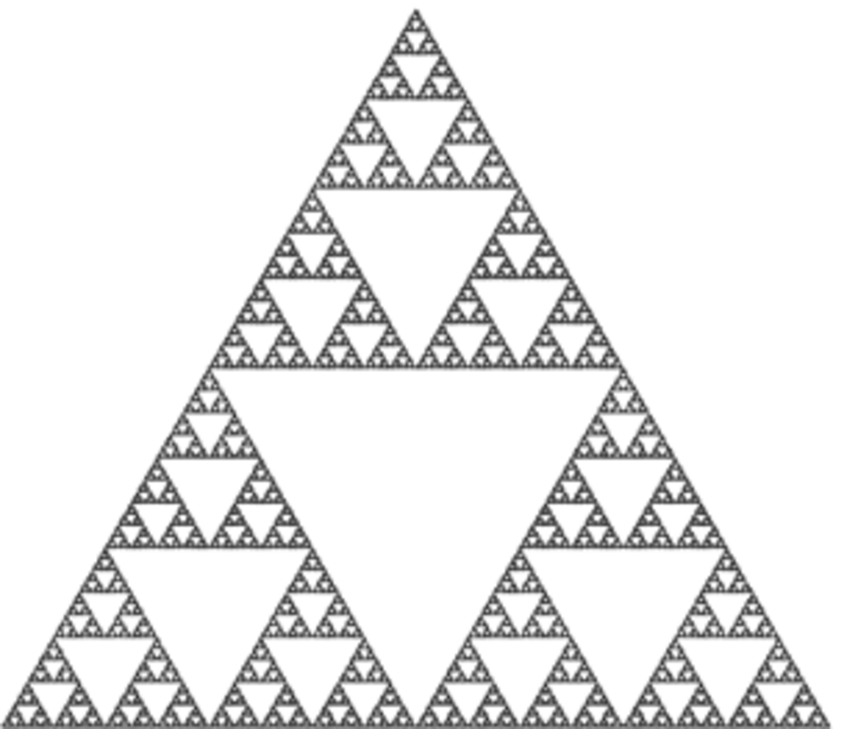} 
    \includegraphics[ width=2.2 in]{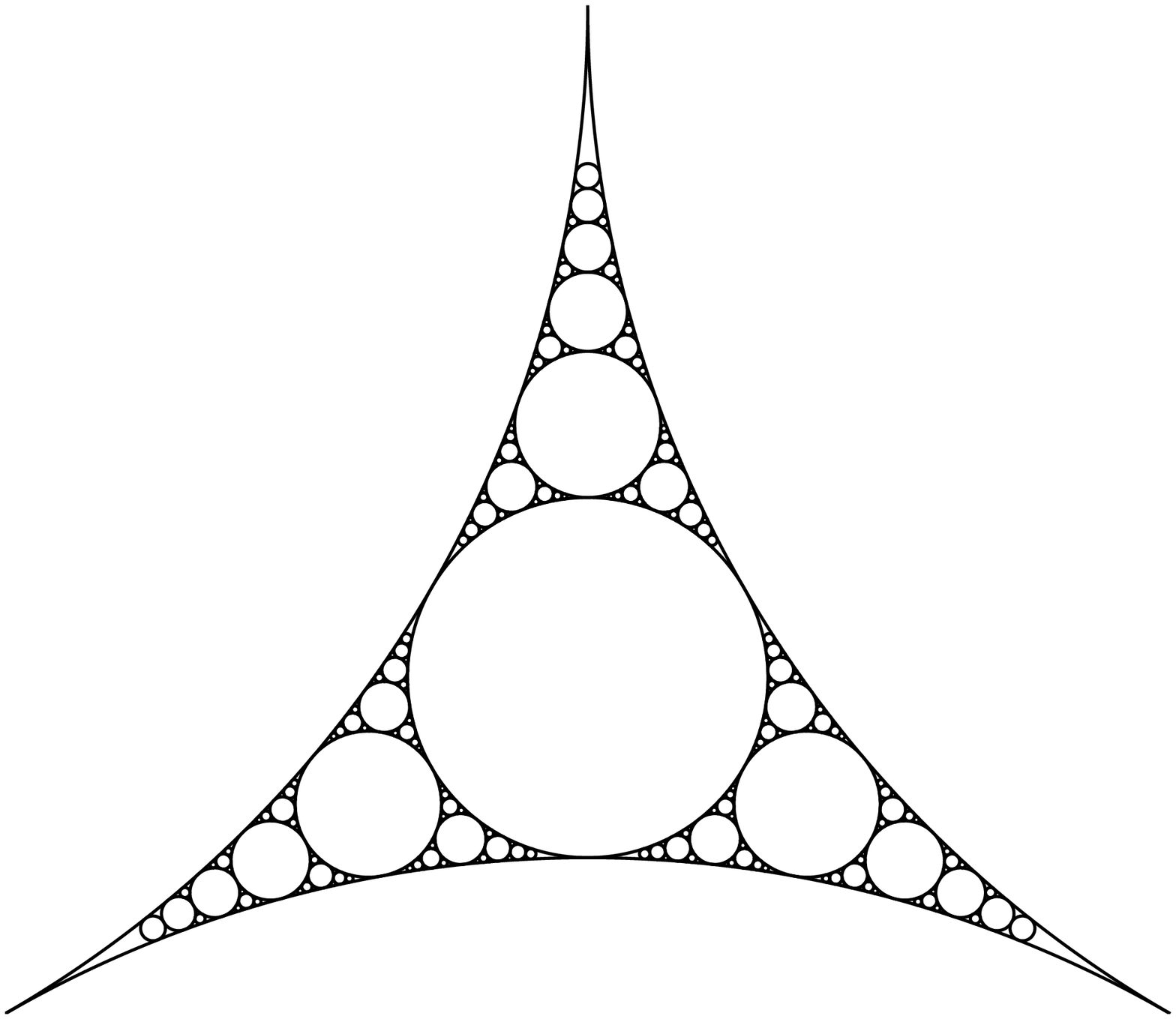} }
   \caption{Sierpinski Gasket {\bf SG} and Apollonian Circle Packing {\bf CP}}
   \label{fig. 7}
\end{figure}

The Sierpinski gasket is obviously homeomorphic to part of an Apollonian circle packing.
 The two figures in Figure 7 are called \emph{fractals} because their
Hausdorff dimension is fractional. Roughly speaking,  since {\bf SG} is made up of three copies of itself, with each copy having
length and width half as big, the dimension of {\bf SG} is the solution of $2^d=3$ (namely $d=\ln 3/\ln 2$).
However, Hausdorff dimension is not a
topological invariant and, indeed, the dimension of {\bf CP} has been found to be $1.305688 \pm 10^{-6}$ by McMullen \cite{McM}
but is not known with complete precision (see Graham et al. [16]).

A {\it local cut point} $x$ in {\bf SG} or in {\bf CP} is an element for which there exists a connected neighborhood $U$ of $x$ for which $U-\{x\}$ is disconnected.   In {\bf SG}, such points are of the form
$\sum_j 2^{-j}P_{f(j)}$ where the corresponding function $f(j)$ is constant for all sufficiently large $j$ or, equivalently, the point can be represented by more than one $f$.

The author's original motivation for this study was the extension of continued fractions on an interval (and their interpretation in terms of Ford circles) to a development of continued fractions on {\bf SG} where, it was hoped, that a type of ``Ford sphere" could be attached to each local cut point.   It is indeed possible for ${\bf CP}$.

Given three mutually tangent circles $C_1, C_2, C_3$ in ${\Bbb C}$ with respective  tangency points $\{w_{ij}\}:=C_i\cap C_j$,  there is a unique set of three spheres $S_{12}, S_{13}, S_{23}$ that are tangent to ${\Bbb C}$ at $w_{12}, w_{13}, w_{23}$ respectively.   By Lemma 2 of \cite{ComplexDCT}, if $c_i$ is the curvature $C_i$ for $i=1,2,3$,  then the curvature of $S_{ij}$ equals $c_i+c_j$.  

By Theorem 5.4, it is clear that it is possible to assign a sphere to each local cut point of the Seirpinski gasket and, in general, to each tangency point of a ``weak" circle packing such as
an Apollonian super-packing (see \cite{GLMWYii}).   The (0,0,1,1) super-packing is pictured below (from Figure 4 of \cite{GLMWYii}).  

\begin{figure}[htbp] 
   \centering
   \includegraphics[ width=3.5 in]{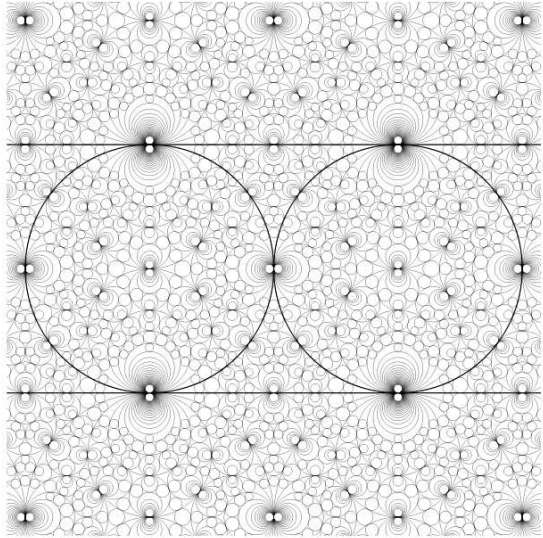} 
   \caption{Part of the (0,0,1,1) super-packing}
   \label{fig. 7}
\end{figure}

\section{Ford spheres:  general case}

A complex quadratic number field is a field of the form $F(D):={\Bbb Q}(\sqrt{-D})$.   The set of algebraic integers of $F(D)$ is  ${\cal O}(D):={\cal O}_{F(D)}={\Bbb Z}[\sigma]$ where
$\sigma=(1+\sqrt{-D})/2$ if $D\equiv 3 \pmod 4$ or $\sigma=\sqrt{-D}$ if $D\equiv 1,2 \pmod 4$ (see, for example, \cite{IR}, \cite{L}).  It follows that
$$|m+n\sigma|^2=\begin{cases} m^2+mn+\frac{D+1}4 n^2 \text{ if } D\equiv 3\pmod 4\\
m^2+Dn^2\text{ if } D\equiv 1,2\pmod 4\end{cases}$$
is, in all cases, an integer.

From a celebrated theorem of Heegner, Stark, and Baker, for $D>0$, ${\cal O}(D)$ is a principal ideal domain (equivalently, a unique factorization domain) if and only if $D$ is one of the ``Heegner numbers" $\{1,2,3,7,11,19,43,67,163\}$.   Further, it is known that ${\cal O}(D)$ is a Euclidean domain (equivalent, in this case to norm-Euclidean) if and only if $D$ is one of the first five Heegner numbers;   i.e., $D\in\{1,2,3,7,11\}$ (see  \cite{L}).    Gauss first posed the question, open to this day, of whether there are infinitely many \emph{ real } quadratic number fields (where $D<0$) with class number 1.   See, for example, \cite{L} for a nice account of Euclidean domains and \cite{IR} for a more general account of quadratic number fields.

For $D\in \{1,2,3,7,11,19,43,67,163\}$, let $\sigma$ be such that ${\cal O}(D)={\Bbb Z}[\sigma]$   
and define Ford spheres in this case to be
$${\cal P}_{\sigma}:=\{S_{\alpha,\beta}:  \alpha,\beta\in{\Bbb Z}[\sigma],  \alpha\perp\beta\}$$
 which thus forms an array of spheres with non-overlapping interiors and such that every sphere is tangent to many others.  

As for Ford circles and the Ford spheres of Sections 2 and 5, we shall show that ${\cal P}_{\sigma}$ is maximal.

\begin{lemma}  For all $z\in{\Bbb C}$, 
$$\inf\{|\beta z-\alpha|: \alpha,\beta\in{\Bbb Z}[\sigma]\}=0.$$ \end{lemma}
\begin{proof}    We may then define a ``floor function" for ${\Bbb C}$:  
$$\lfloor x+\sigma y\rfloor:=\lfloor x \rfloor+ \sigma\lfloor y \rfloor$$
and a corresponding ``fractional part" function
$$\{x+\sigma y\}:=\{x\}+\sigma \{y\}.$$
Given positive integer $N$, divide the parallelogram $R:=\{x+y\sigma:  x,y\in[0,1)\}$ into $N^2$ disjoint congruent parallelograms similar to $R$.    

Clearly, the lemma holds for $z\in {\Bbb Q}(\sigma)$.  Fix $z\in{\Bbb C}-{\Bbb Q}(\sigma)$.  For distinct integers $j$ and $k$,  $\{jz\}\neq \{kz\}$ since, otherwise, $(j-k)z\in {\Bbb Z}[\sigma]$.   Hence, for $j=1,...,n^2+1$, the numbers $\{jz\}$ are distinct and so, by the pigeonhole principle, one of the little parallelograms contains distinct $\{jz\}$ and $\{kz\}$.   
Hence, for some $\alpha_j,\alpha_k\in{\Bbb Z}[\sigma]$, 
$$|(jz-\alpha_j)-(kz-\alpha_k)|=|\{jz\}-\{kz\}|<diam(R)/n$$
and the lemma follows.   \end{proof}

\begin{lemma}  ${\cal P}_{\sigma}$ is maximal. \end{lemma}
\begin{proof}  
Note that any two Ford spheres $S_{\alpha,\beta},S_{\gamma,\delta}$ have disjoint interiors (since $|\alpha\delta-\beta\gamma|\ge 1$).  Furthermore, since $\alpha\perp\beta$, there exist $\gamma,\delta$ such that $|\alpha\delta-\beta\gamma|= 1$ and so  $S_{\alpha,\beta}||S_{\gamma,\delta}$.   Hence no Ford sphere can be enlarged.   

     Suppose there exist $z,r$ so that  for all $a,b$, $S(z,r)^{\circ}\cap S_{\alpha,\beta}= \emptyset$.  Obviously $z\not\in{\Bbb Q}(\sigma)$  and thus, by the previous lemma, the set $\{\beta z-\alpha: \alpha\in{\Bbb Z}[\sigma], \beta\in{\Bbb Z}[\sigma]^+ \}$ is dense in ${\Bbb C}$  and so, for any fixed $r$, there exist $\alpha,\beta$ such that 
$$\left|z-\frac {\alpha}{\beta}\right|<\frac{\sqrt{2r}}{|\beta|}.$$
This implies $S(z,r)^{\circ}\cap S_{\alpha,\beta}^{\circ}\neq\emptyset$, a contradiction.  \end{proof}

A barycentric representation of these spheres are based on solutions of the equation:
$$\begin{cases} 
ab+ac+bc+(a+b+c)m=\frac{D-3}4 m^2 &\text{ if $D\equiv 3 \pmod 4$}\\
ab+ac+bc=Dm^2 &\text{ otherwise.}
\end{cases} \eqno{(11)}$$
Let ${\cal B}_{\sigma}:=\{\langle a+m\xi, b+m\xi, c+m\xi\rangle:  (a,b,c)\in{\Bbb Z}_{\perp}^3 \text{ is a solution of (11)}\}$
where 
$$\xi=\begin{cases} &(\sqrt 3-\sqrt D)/\sqrt{12} \text{ if $D\equiv 3 \pmod 4$}\\
&\sqrt D/\sqrt 3\text{ otherwise.} \end{cases}$$
Let $\mu(z):=\begin{pmatrix}
\omega&0\\
\omega&1\\
\end{pmatrix} (z)$ be the M\"obius transformation taking the points $0,1,\infty$ to $0, 1+\omega, 1$ respectively.

\begin{theorem}  $\mu(P_{\sigma})={\cal B}_{\sigma}$.  \end{theorem}

To prove this, we first prove several lemmas.   

\begin{lemma} If $\overline\alpha\beta=s+it$ then  
$$\mu(S_{\alpha,\beta})=\langle a+m/\sqrt 3, b+m/\sqrt 3, c+m/\sqrt 3\rangle$$
where $a=|\beta|^2-s$,  $b=|\alpha|^2-s$, $c=s$, and $ab+ac+bc=m^2=t^2$.   \end{lemma}

\begin{proof}  Let $a=|\beta|^2-s$,  $b=|\alpha|^2-s$, $c=s$, and $m=t$.  
Then 
$$ab+ac+bc=|\overline\alpha\beta|^2-s(|\alpha|^2+|\beta|^2-2s)=|\overline\alpha\beta|^2-s^2=t^2=m^2.$$
Note that 
$$\omega\alpha\overline\beta=-s/2+\sqrt 3 t/2+i(\sqrt 3s/2+t/2)$$
and thus 
$$\begin{aligned}&(a+m/\sqrt 3)+(b+m/\sqrt 3)+(c+m/\sqrt 3)=|\alpha|^2+|\beta|^2-s+t\sqrt3\\&=|\alpha|^2+|\beta|^2+2\Re(\omega\alpha\overline\beta)=|\omega\alpha+\beta|^2,\end{aligned}$$
and
$$\begin{aligned}&(b+m/\sqrt 3)+(c+m/\sqrt 3)(1+\omega)=|\alpha|^2-s/2+\sqrt 3 t/2+i(\sqrt 3s/2+t/2)\\&=|\alpha|^2+\omega\alpha\overline\beta=\omega\alpha\overline{(\omega\alpha+\beta)}.\end{aligned}$$
Hence, by Proposition 3.2 
$$\begin{aligned}   \mu(S_{\alpha,\beta})&=S_{\omega\alpha, \omega\alpha+\beta}=S\left(\dfrac{\omega\alpha}{\omega\alpha+\beta}, \dfrac 1{2|\omega\alpha+\beta|^2}\right)\\&=S\left(\dfrac{\omega\alpha\overline{(\omega\alpha+\beta)}}{|\omega\alpha+\beta|^2}, \dfrac 1{2|\omega\alpha+\beta|^2}\right)\\&=S\left(\dfrac{b+m/\sqrt 3+(c+m/\sqrt 3)(1+\omega)}{a+b+c+m\sqrt 3}, \dfrac 1{2(a+b+c+m\sqrt 3)}\right)\\&=\langle a+m/\sqrt 3, b+m/\sqrt 3, c+m/\sqrt 3\rangle.\end{aligned}$$
\end{proof}

\begin{lemma} If $\mu(S_{\alpha,\beta})=\langle a+m/\sqrt 3, b+m/\sqrt 3, c+m/\sqrt 3\rangle$ and $\mu(S_{\gamma,\delta})=\langle a'+m'/\sqrt 3, b'+m'/\sqrt 3, c'+m'/\sqrt 3\rangle$ then
$$|\alpha\delta-\beta\gamma|^2=ab'+ac'+ba'+bc'+ca'+cb'-2mm'.$$\end{lemma}

\begin{proof}
By Lemma 6.4,  the hypothesis implies $a=|\beta|^2-s$, $b=|\alpha|^2-s$, $c=s$, $m=t$ where $\overline\alpha\beta=s+it$
and $a'=|\delta|^2-s'$, $b'=|\gamma|^2-s'$, $c'=s'$, $m'=t'$ where $\overline\gamma\delta=s'+it'$.  
Then $$\begin{aligned} a(b'+c')+&b(a'+c')+c(a'+b')-2mm'\\&=(|\beta|^2-s)|\gamma|^2+(|\alpha|^2-s)|\delta|^2+s(|\gamma|^2+|\delta|^2-2s')-2tt'\\&=|\alpha\delta|^2+|\beta\gamma|^2-2(ss'+tt')\\&=|\alpha\delta|^2+|\beta\gamma|^2-2\Re(\overline\alpha\beta\gamma\overline\delta)=|\alpha\delta-\beta\gamma|^2.\end{aligned}$$
\end{proof}

\noindent{\it Proof of Theorem 6.3.}  Suppose $D\equiv 3\pmod 4$,  $\sigma=(1+i\sqrt{D})/2$.    Given $\alpha,\beta\in{\Bbb C}$,  let $s+it=\overline{\alpha}\beta$,
$A:=|\beta|^2-s+t/\sqrt D$,  $B:=|\alpha|^2-s+t/\sqrt D$,   $C:=s+t/\sqrt D$, $M:=-2t/\sqrt D$, and $\xi:=(\sqrt 3-\sqrt D)/\sqrt{12}$.  By Lemma 6.4,
$$\mu(S_{\alpha,\beta})=\langle A+M\xi, B+M\xi, C+M\xi\rangle.$$
By Lemma 6.5, if $\mu(S_{\alpha,\beta})=\langle A+M\xi, B+M\xi, C+M\xi\rangle$ and $\mu(S_{\gamma,\delta})=\langle A'+M'\xi, B'+M'\xi, C'+M'\xi\rangle$ then 
$$\begin{aligned}|\alpha\delta-\beta\gamma|^2&=AB'+AC'+BA'+BC'+CA'+CB'\\&+M(A'+B'+C')+M'(A+B+C)-\frac{D-3}2 MM'. \end{aligned}\eqno{(12)}$$

Given two distinct spheres in ${\cal B}_{\sigma}$, say $\langle A+M\xi, B+M\xi, C+M\xi\rangle$ and $\langle A'+M'\xi, B'+M'\xi, C'+M'\xi\rangle$, their images under $\mu^{-1}$ are of the form $S_{x,y}, S_{u,v}$ respectively for some complex $x,y,u,v$.   Equation (12) implies $|xv-yu|^2$ is a (positive) integer and therefore the collection of spheres
${\cal B}_{\sigma}$ is normal (i.e., no two spheres have intersecting interiors).   

If $\alpha, \beta\in{\Bbb Z}[\sigma]$, then $|\alpha|^2$ and $|\beta|^2$ are integers.   Also, $\overline\alpha\beta=s+it$ where $2s, 2t/\sqrt D\in{\Bbb Z}$ and $2s\equiv 2t/\sqrt D\pmod 2$.  Then $A,B,C$, and $M$ are integers.  If $\alpha\perp\beta$ then there exist $\gamma,\delta$ such that
$|\alpha\delta-\beta\gamma|=1$  and so, by (12),  
$$\begin{aligned} AB'+AC'&+BA'+BC'+CA'+CB'\\&+M(A'+B'+C')+M'(A+B+C)-\frac{D-3}2 MM'=1\end{aligned}$$
for some integers $A',B',C',M'$ and therefore
$A,B,C,M$ are relatively prime.  Hence $\mu(S_{\alpha,\beta})\in{\cal B}_{\sigma}$.  By the maximality of ${\cal P}_{\sigma}$,  the theorem holds.

Suppose $D\equiv 1,2\pmod 4$,  $\sigma=i\sqrt D$.    Given $\alpha,\beta\in{\Bbb C}$,  let
$A:=|\beta|^2-s$,  $B:=|\alpha|^2-s$,   $C:=s$, $M:=t/\sqrt D$, and $\xi:=\sqrt D/\sqrt{3}$.  By Lemma 6.4, 
$$\mu(S_{\alpha,\beta})=\langle A+M\xi, B+M\xi, C+M\xi\rangle.$$
By Lemma 6.5, if $\mu(S_{\alpha,\beta})=\langle A+M\xi, B+M\xi, C+M\xi\rangle$ and $\mu(S_{\gamma,\delta})=\langle A'+M'\xi, B'+M'\xi, C'+M'\xi\rangle$ then 
$$|\alpha\delta-\beta\gamma|^2=AB'+AC'+BA'+BC'+CA'+CB'-2DMM'. \eqno{(13)}$$

Given two distinct spheres in ${\cal B}_{\sigma}$, say $\langle A+M\xi, B+M\xi, C+M\xi\rangle$ and $\langle A'+M'\xi, B'+M'\xi, C'+M'\xi\rangle$, their images under $\mu^{-1}$ are of the form $S_{x,y}, S_{u,v}$ respectively for some complex $x,y,u,v$.   Equation (13) implies $|xv-yu|^2$ is a (positive) integer and therefore the collection of spheres
${\cal B}_{\sigma}$ is normal (i.e., no two spheres have intersecting interiors).   

If $\alpha, \beta\in{\Bbb Z}[\sigma]$, then $|\alpha|^2$ and $|\beta|^2$ are integers.   Also, $\overline\alpha\beta=s+it$ where $s,t\in{\Bbb Z}$.  Then $A,B,C$, and $M$ are integers.  If $\alpha\perp\beta$ then there exist $\gamma,\delta$ such that
$|\alpha\delta-\beta\gamma|=1$  and so, by (13),  $$AB'+AC'+BA'+BC'+CA'+CB'-2DMM'=1$$ for some integers $A',B',C',M'$ and therefore
$A,B,C,M$ are relatively prime.  Hence $\mu(S_{\alpha,\beta})\in{\cal B}_{\sigma}$.  By the maximality of ${\cal P}_{\sigma}$,  the theorem holds.
$\square$ 
\vskip .1 in

 As in the proof of Theorem 6.3, since $\alpha\perp\beta$ implies $A,B,C,M$ are relatively prime,  we have a test for relative primality in ${\Bbb Z}[\sigma]$. 
  
 \begin{corollary}  $\alpha\perp\beta$ in ${\Bbb Z}[\sigma]$ if and only if 
 $$\gcd(|\alpha|^2, |\beta|^2, s, t/\sqrt D)=1 \text{ if $D\equiv 1,2\pmod 4$}$$
 $$\gcd(|\alpha|^2, |\beta|^2, s-t\sqrt D, 2t/\sqrt D)=1 \text{ if $D\equiv 3\pmod 4$}$$
 where $s+it=\overline\alpha\beta$.\end{corollary}
 
 %%New stuff

We now find an algorithm for finding the relatively prime integer solutions of (10) (and thus the elements of ${\cal B}_{\sigma}$ and ${\cal P}_{\sigma}$).    For any quadratic polynomial $f(x)$, the ``secant addition" 
$$x\oplus y:=\dfrac{xf(y)-yf(x)}{f(y)-f(x)}$$
is associative and, in fact, if $f$ has roots $u$ and $v$, then $(({\Bbb R}-\{u,v\})\cup\{\infty\}, \oplus, \infty)$ is an abelian group  (see \cite{NorSecant}).
The ring of algebraic integers ${\Bbb Z}[\sigma]$ has characteristic polynomial
 $$f(x):=\begin{cases}
x^2-x+\frac{D+1}4 &\text{ if $D\equiv 3 \pmod 4$}\\
x^2+D &\text{ otherwise.}
\end{cases}$$
The corresponding secant addition is defined by
 $$x\oplus y:=\begin{cases}
 \dfrac{xy-(D+1)/4}{x+y-1} &\text{ if $D\equiv 3 \pmod 4$}\\
\dfrac{xy-D}{x+y} &\text{ otherwise.}
\end{cases}$$
It follows that $x\oplus y\oplus z=\infty$ if and only if $(-x,-y,-z)$ is a solution of (11).   
Hence, to find (all) $(a,b,c)\in{\Bbb Z}_{\perp}^3$ that solves (11):
\begin{itemize}
\item  Choose $x,y\in{\Bbb Q}$,
\item  Calculate $z:=(x\oplus y)^{-1}$,
\item  Let $m$ be the least positive integer so that $m(xy+xz+yz)\in{\Bbb Z}$,
\item  Let $a=-mx,b=-my, c=-mz.$
\end{itemize}

Just as there is a recursive geometric construction of $B_{\sigma}$ based on a tetrahedron when $D=3$ and based on an octahedron when $D=1$, we believe that for all five cases $D=1,2,3,7, 11$ corresponding to ${\Bbb Z}[\sigma]$ being a Euclidean domain, there is a recursive construction based on a polyhedron.     Figure 9 shows the conjectured polyhedra with respective discriminants 3,4,7,8,11 (i.e., $D=3,1,7,2,11$ respectively).    Computer experimentation supports the conjectured relation, as do some results in the literature (see  \cite{Y}).   

\begin{figure}[htbp] 
   \centering{
   \includegraphics[ width=.8 in]{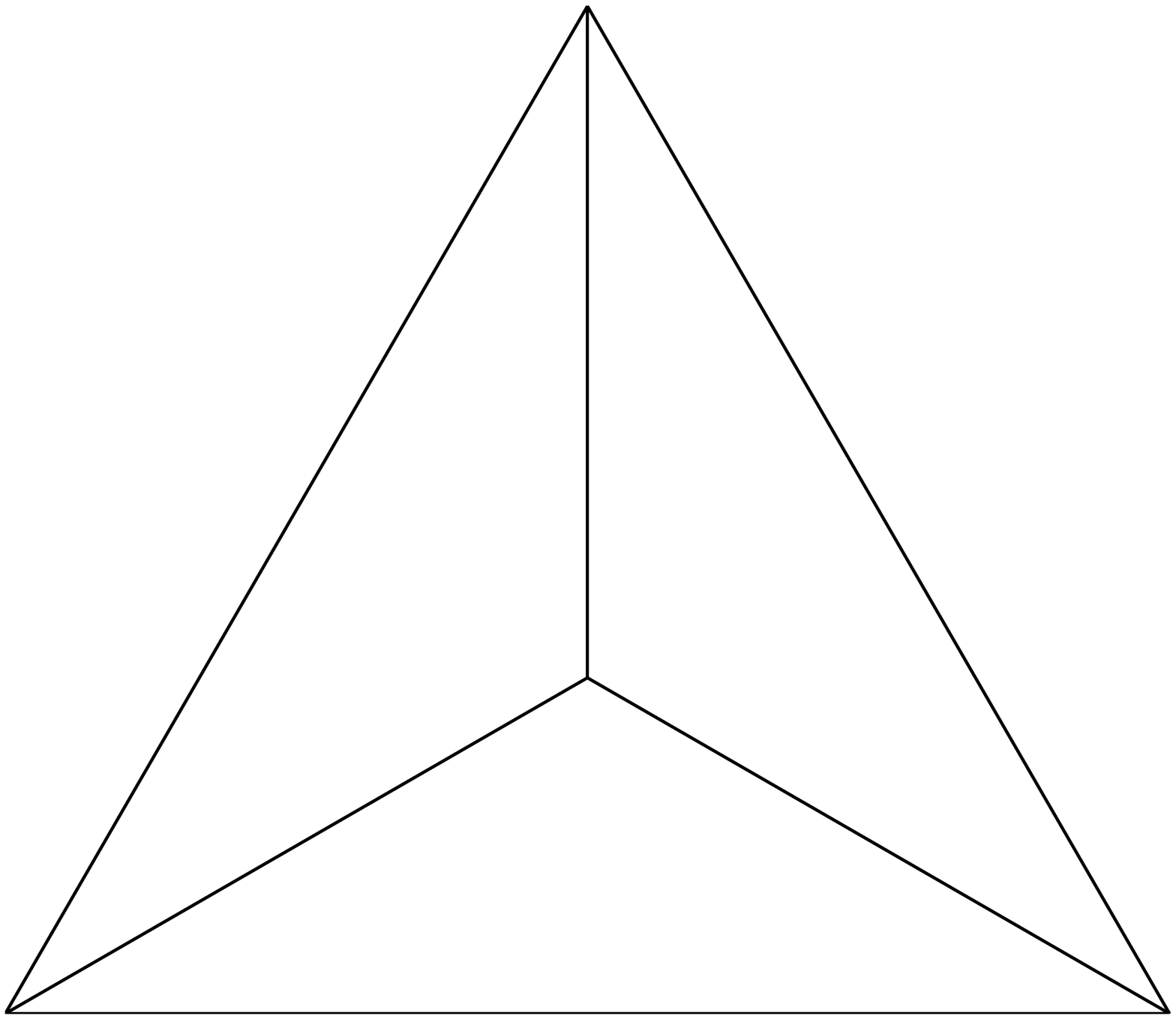} 
   \includegraphics[ width=.8 in]{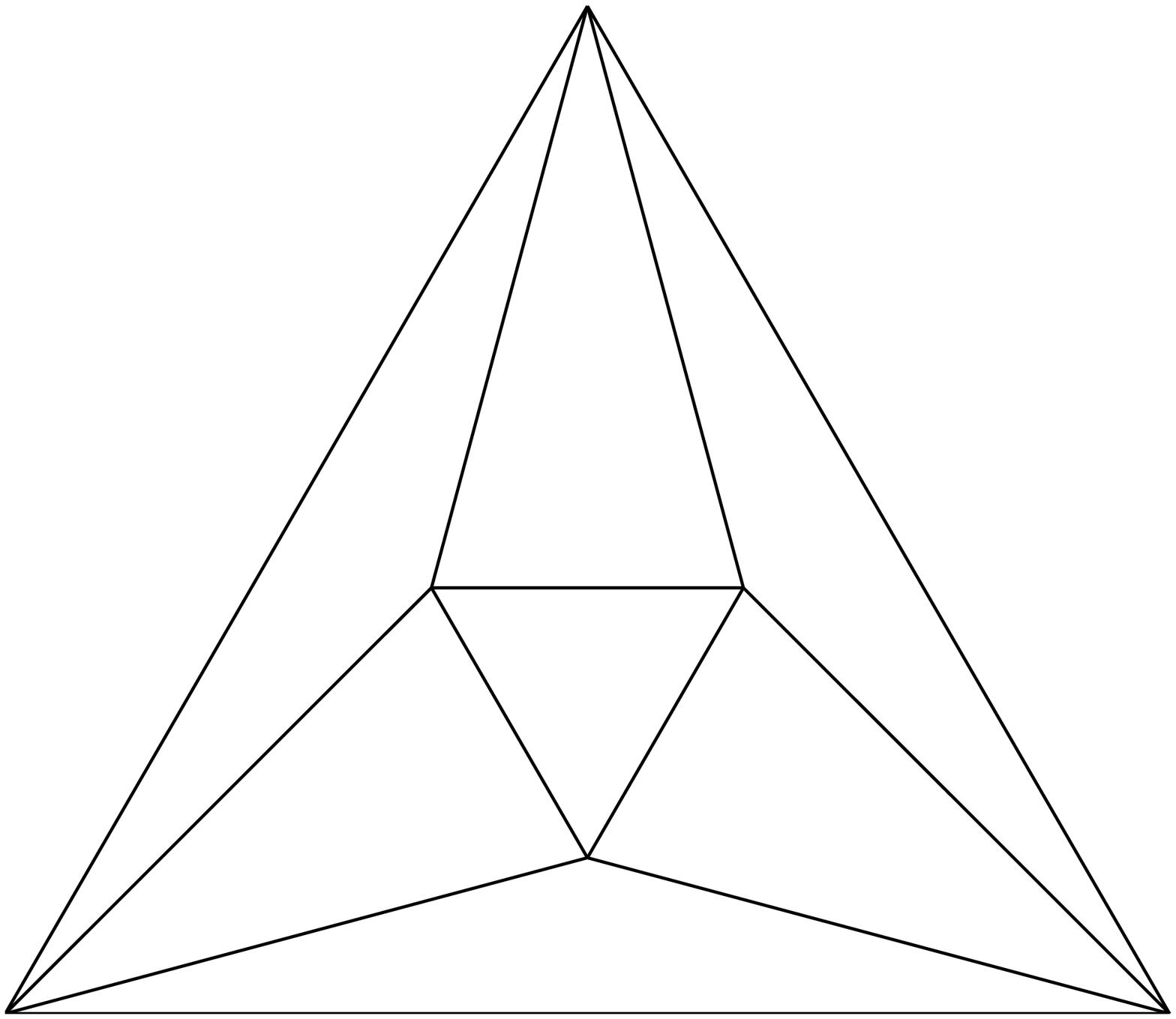} 
   \includegraphics[ width=.8 in]{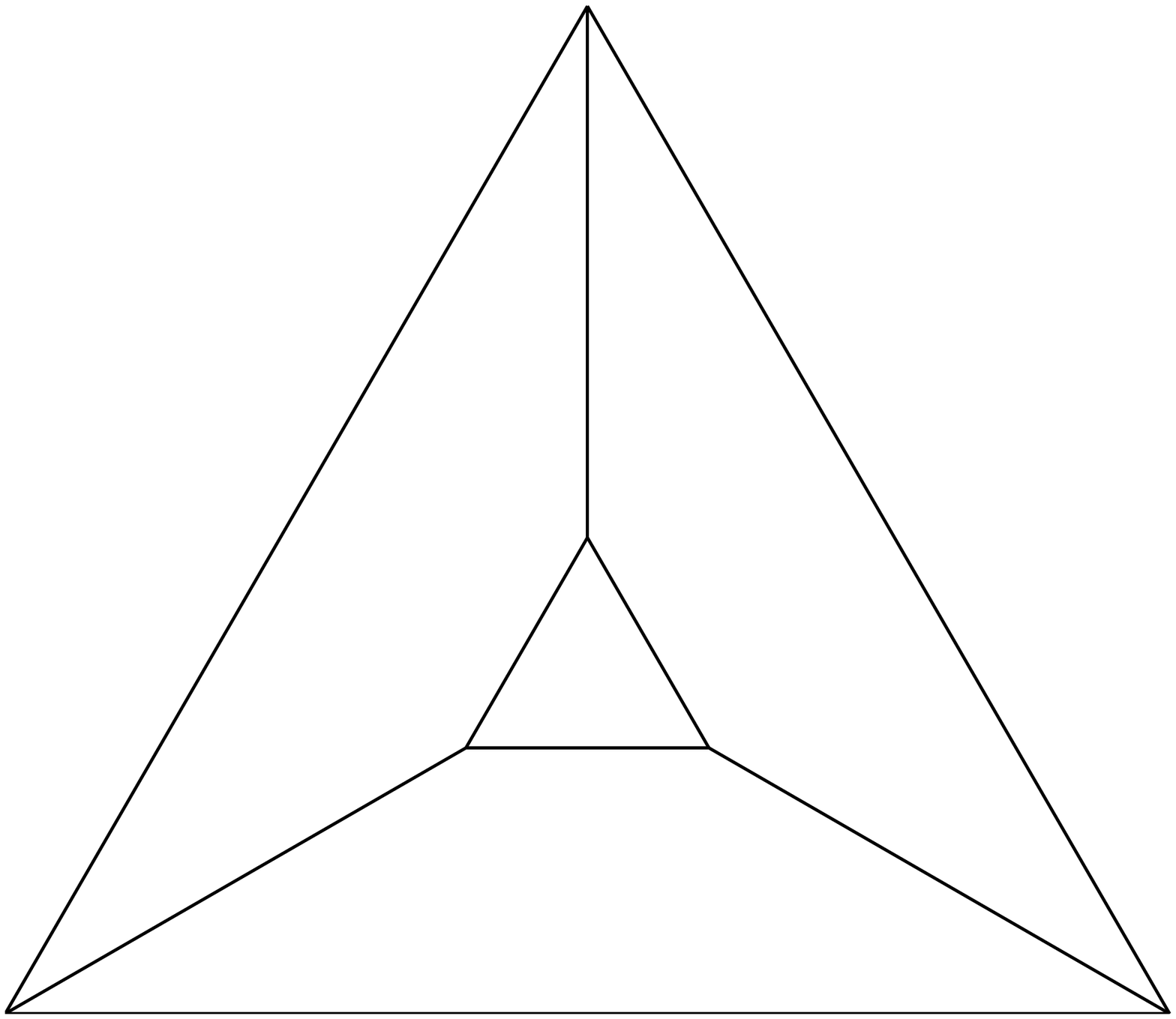}
    \includegraphics[ width=.8 in]{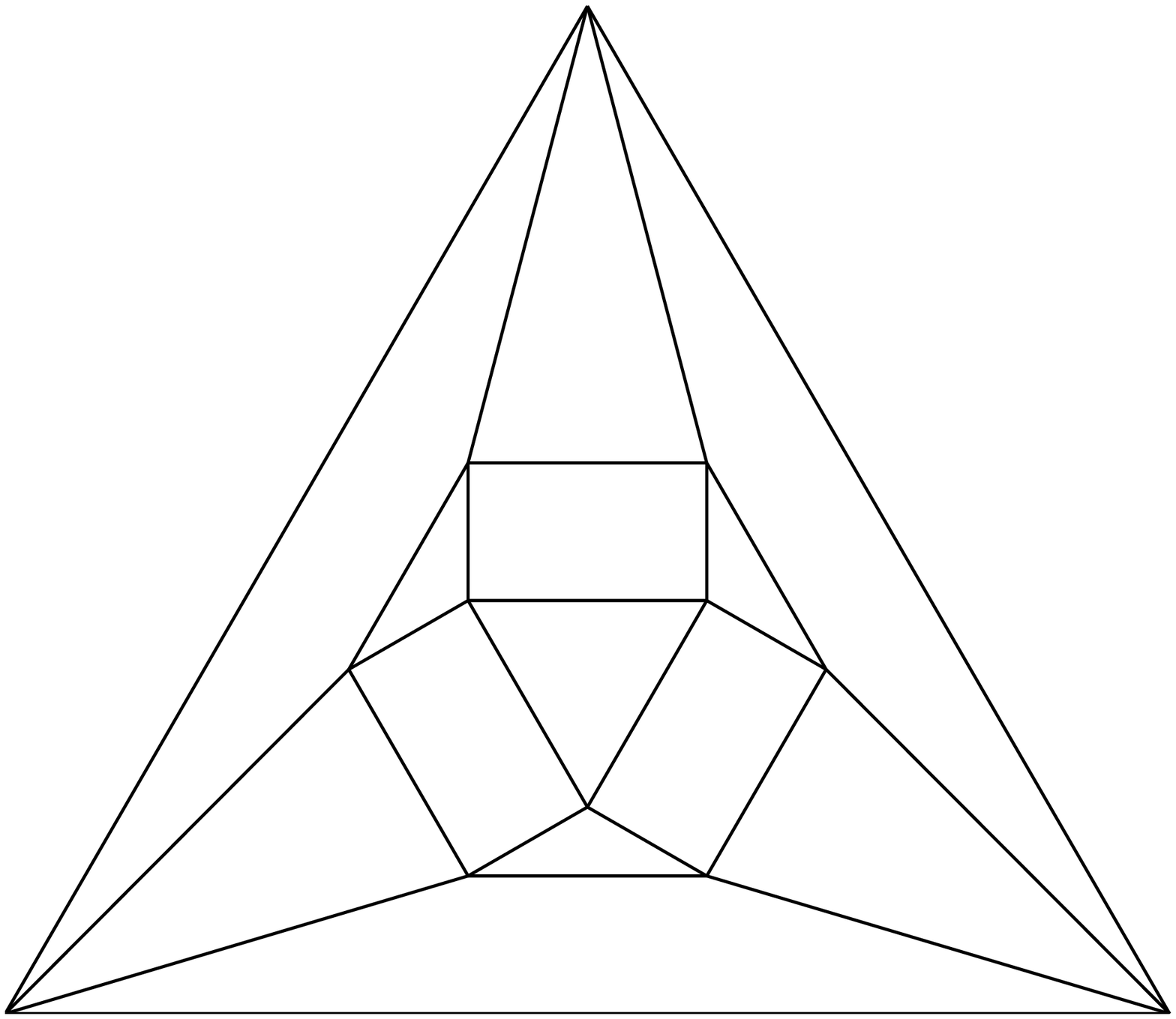} 
    \includegraphics[ width=.8 in]{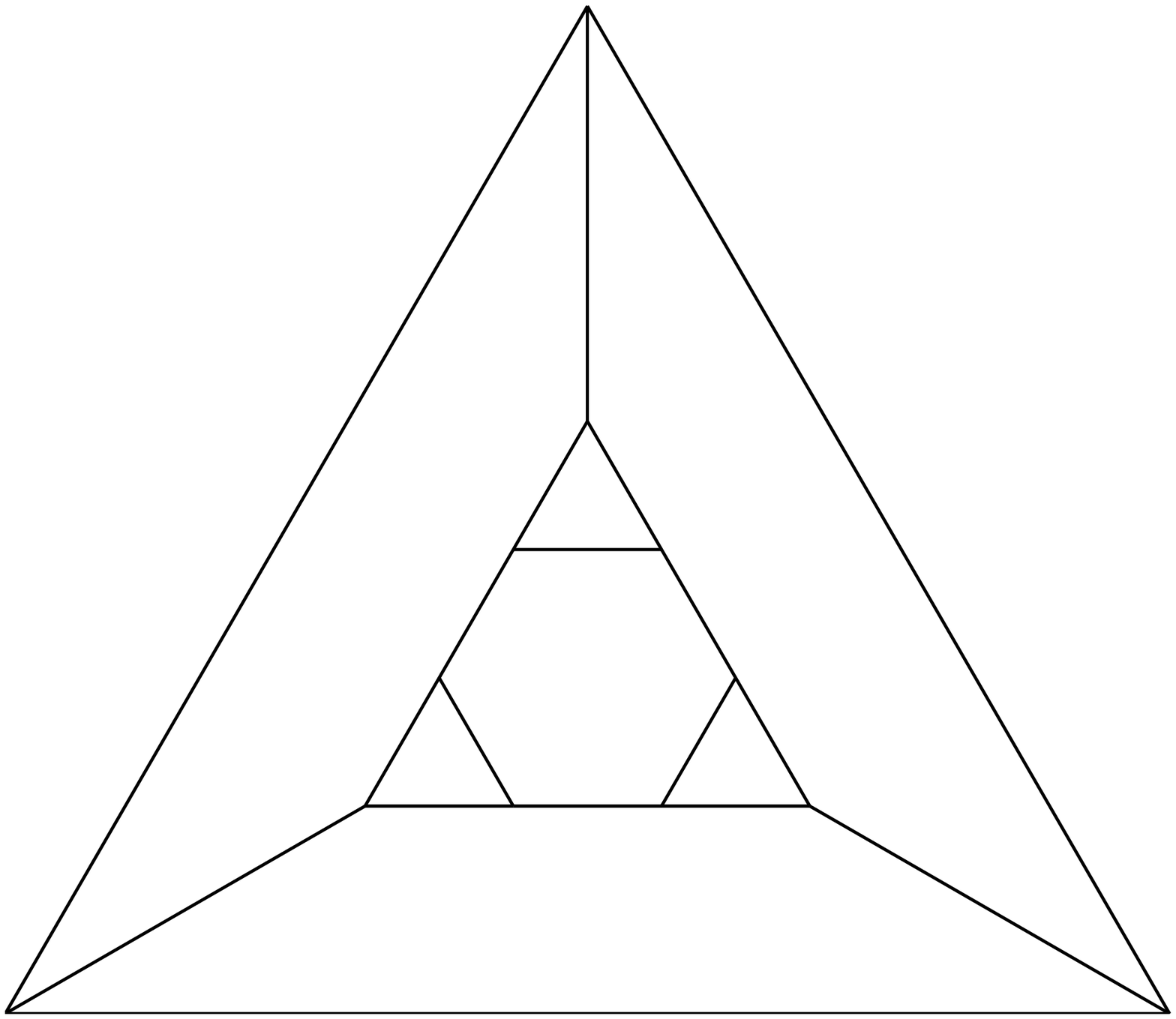} 
    }
   \caption{Polyhedra corresponding to the five imaginary quadratic Euclidean domains.}
   \label{fig. 7}
\end{figure}

\section{Conclusion}

In this section we attempt to connect the material of this paper to past, present, and future research.   

The main direction of this paper was to understand and generalize Ford circles by, first, developing three different parameterizations of them (${\cal P,G,B}$);  the third being apparently new.  There are therefore three different directions Ford circles to be generalized.   We first consider the geometric, followed by the algebraic, and then ``barycentric".

The initial motivation for this paper was to extend continued fractions in a new way.   Since
continued fractions are closely related to the geometry of Ford circles (a topic not developed in this paper but which appears, implicitly at least, in \cite{Se}), any geometric generalization of Ford circles could lead to a generalization of continued fractions.    A natural extension of the interval is to the Sierpinski gasket  (or any of its homeomorphic images)  see Figure 7, and a generalization of continued fractions to that space might be new.   Assigning spheres to the local cut points of {\bf CP} seemed to be promising and indeed this is the type of sphere developed in Section 5.

\begin{problem}  Develop a theory of continued fractions on {\bf CP}.  \end{problem}

Three-dimensional analogues of Ford circles have been studied previously by Hellegouarch \cite{H1, H2},  Pickover \cite{Pick}, and Rieger \cite{R1,R2} but a unifying treatment of them has not been done.   This paper, although not explicitly a survey of the topic, was meant to cover the topic more comprehensively.      In particular, we cover the cases corresponding to ${\Bbb Q}(i)$, ${\Bbb Q}(\omega)$ and, to a lesser extent, ${\Bbb Q}(\sigma)$ for which ${\Bbb Z}[\sigma]$ is a UFD.
\begin{problem} Develop the connection between continued fractions and the the family ${\cal P}_{\sigma}$ of spheres. \end{problem}

Diophantine approximation is a major theme of continued fractions and can be understood geometrically in terms of Ford circles;  see \cite{C}, p. 30.  Schmidt \cite{Sch1, Sch2, Sch3, Sch4} developed the theory of Diophantine approximation for four of the imaginary quadratic number fields for which the ring of integers is a Euclidean domain.    For a rational number $q:=a/b$ ($a,b$ in lowest terms), let $C(q):=C_{a,b}$.  
The ``parent algorithms" of Sections 2, 4, and 5 leads to an algorithm for ``climbing up" the circles or spheres.    This process, for the circles at least, is dual to the Euclidean algorithm:  if $[0; a_1,a_2,...,a_n]$ represents a rational number $q:=a/b$, then the Gauss map of it (which is tantamount to running the Euclidean algorithm one step) satisfies $\{1/q\}=[0; a_2,..., a_n]$  whereas the parents of $C(q)$ are $C(q')$ and $C(q'')$ where $q'=[0; a_1,..., a_{n-1}]$ and $q''=[0; a_1,...,a_n-1]$.  Hence, the Euclidean property for ${\Bbb Z}$ is equivalent to the property that every Ford circle of radius less than $1/2$ touches a larger one.    This is also a consequence of having a recursive geometric procedure (that adds smaller circles) which generates all of the Ford circles.  

It is reasonable then to consider the tetrahedral and octahedral geometric recursions of Sections 4 and 5 as ways to ``see" that ${\Bbb Z}[i]$ and ${\Bbb Z}[\omega]$ are Euclidean domains.  
\begin{problem}  It is well known that there are exactly five imaginary quadratic fields whose rings of integers are Euclidean. Are there then only five polyhedra that define a recursive geometric procedure that gives all the Ford spheres parameterized by these fields?   \end{problem}

A known theorem states that the integers of an imaginary quadratic field form a UFD if and only if it satisfies a multistage Euclidean algorithm (Proposition 3.2 of \cite{L}).  A celebrated theorem of Heegner, Stark, and Baker (all independent) showed that there are precisely 19 such fields.   

\begin{problem}  What polyhedra form contact graphs for spheres and whose iterates form all the Ford spheres for the 19 UFD's?  \end{problem} 

\noindent Work in this direction seems to have been done by Yasaki \cite{Y}.

Given any circle packing in ${\Bbb C}$, normal spheres can be attached to each point of tangency of these circles (assign curvature to the sphere equal to the sum of curvatures of the two circles;  see \cite{ComplexDCT}) so as to form an array of non-overlapping normal spheres.    By the Koebe-Andreev-Thurston theorem, every planar graph can be realized as the contact graph of some circle packing and thus give rise to the contact graph of normal spheres.   
\begin{problem}  Every finite contact graph of spheres generates a recursive geometric process leading to an infinite family of spheres whose points of tangency with ${\Bbb C}$ form a set with some type of Euclidean property;  can this idea be developed and/or can such a set of be algebraically defined? \end{problem}

Although we expanded ${\cal P}$ to ${\cal P_{\sigma}}$ in Sections 4,5, and 6,  for imaginary quadratic fields with class number 1,  it seems possible to extend further.
\begin{problem}  Investigate ${\cal P_{\sigma}}$ for $\sigma$, say, a cubic algebraic number.  \end{problem}

\begin{problem}  Solutions of equation (11) parameterized Ford circles ${\cal P}_{\sigma}$;   do more general equations parameterize other interesting families of spheres? \end{problem}

\begin{problem} A conjecture by Gauss states that there are infinitely many real quadratic fields of class number 1 (i.e., ring of integers is a UFD).   Are there definitions of Ford circles that are relevant to this conjecture?    \end{problem}

 For example,  ${\Bbb Q}(\sqrt 2)$ is a UFD.   Given $a,b,c,d\in {\Bbb Z}[\sqrt 2]$, consider the recursive geometric procedure for normal circles: $$\{C_{a,b},C_{c,d}\}\mapsto \{C_{a,b}, C_{a\sqrt 2+c, b\sqrt2+d},C_{a+c\sqrt 2, b+d\sqrt 2}, C_{c,d}\}$$
The points where these circles intersect ${\Bbb R}$ take on values that form a strict subset of ${\Bbb Z}[\sqrt 2]$.   These circles form a subset of a circle packing investigated by Guettler and Mallows \cite{GM}.

 At the end of Section 6, it was seen Ford spheres parameterized by ${\Bbb Q}(\sigma)$ can also be parameterized by a group $\{(x,y,z)\in{\Bbb{Q}}:  x\oplus y\oplus z=e\}$ where $\oplus$ is the ``secant addition" \cite{NorSecant} associated with 
the minimum polynomial for $\sigma$.  
\begin{problem}  A group of rationals (possibly with $e=\infty$) can be based on secant addition for any quadratic polynomial.   What families of spheres are generated in this way?   Is there a way to understand the group structure geometrically (i.e., in terms of the spheres)?  Considering cubic polynomials in this context,  $\oplus$ is still defined though not associative.   Is there a way to define spheres in this case? \end{problem}

\bigskip

\noindent\textit{Department of Mathematics, 
SUNY, Plattsburgh, NY 12901\\
northssw@plattsburgh.edu}

\end{document}